\documentclass[12pt, oneside]{amsart}
\usepackage[left=3cm,right=2cm,top=3cm,bottom=2cm]{geometry}
\usepackage{graphicx,color}
\usepackage{multirow}
\usepackage{amsmath}
\usepackage[english]{babel}
\usepackage{amssymb}
\usepackage{epsfig}
\usepackage{graphicx,color}
\usepackage[latin1]{inputenc}
\usepackage{latexsym}
\usepackage[normalem]{ulem}
\usepackage{graphpap}
\usepackage{bm}

\usepackage[shortlabels]{enumitem}
\usepackage[pdftex]{hyperref}
\usepackage[pdftex]{hyperref}

\DeclareMathOperator{\rank}{\mathrm{rank}}

\newtheorem{teo}{Theorem}[section]
\newtheorem{cor}{Corollary}[section]
\newtheorem{defi}{Definition}[section]
\newtheorem{obs}{Remark}[section]
\newtheorem{lema}{Lemma}[section]
\newtheorem{prop}{Proposition}[section]
\newtheorem{ex}{Example}[section]
\newtheorem*{conjecture}{Conjecture}

\newcommand{\isEquivTo}[1]{\underset{#1}{\sim}}
\newcommand{\bnu}{\bm{\nu}}

\newcommand{\cP}{\mathcal{P}}
\newcommand{\cA}{\mathcal{A}}
\newcommand{\cB}{\mathcal{B}}
\newcommand{\cM}{\mathcal{M}}

\newcommand{\cO}{\mathcal{O}}
\newcommand{\cE}{\mathcal{E}}

\newcommand{\cT}{\mathcal{T}}
\newcommand{\cK}{\mathcal{K}}
\newcommand{\cG}{\mathcal{G}}

\newcommand{\cU}{\mathcal{U}}

\newcommand{\K}{\mathcal{K}}
\newcommand{\bx}{\bm{\xi}}
\newcommand{\by}{\bm{y}}
\newcommand{\br}{\bm{x}}

\newcommand{\cal}{\mathcal}
\newcommand{\R}{\mathbb{R}}

\newcommand{\NN}{\mathbb{N}}

\title{
	Singularities of 3-parameter line congruences in  $\R^4$ }

\author{$\text{D. Lopes}^\dagger,\; \text{M.A.S. Ruas}^\ddagger\; \text{and}\; \text{I.C. Santos}^*$}

\thanks{\\ ${^\dagger} \text{Supported by INCTMat/CAPES Proc. 88887.510549/2020-00}\\
{^\ddagger}\text{Partially supported by FAPESP Proc. 2019/21181-0 and CNPq
	Proc. 305695/2019-3}\\ 
{^*}\text{Supported by CAPES Proc. PROEX-11365975/D}$}

\begin{document}
\begin{abstract}
	In this paper, we give the generic classification of the singularities of 3-parameter line congruences in $\R^4$. We also classify the generic singularities of Blaschke (affine) normal congruences.
\end{abstract}

\maketitle
\section{Introduction}
In \cite{Monge} Monge was one of the first authors to discuss line congruences in $\R^3$. In recent decades, some papers are dedicated to the study of line congruences from  singularity theory and  differential affine geometry viewpoints (\cite{Barajas}, \cite{Craizer}, \cite{Giblin}, \cite{Izumiya}). There is a particular interest in the Blaschke normal congruences, in the behavior of affine principal lines near an affine umbilic point (\cite{Barajas}) and the behavior of affine curvature lines at isolated umbilic points (\cite{Craizer}). From singularity theory viewpoint, there is a particular interest in the classification of the singularities related to line congruences (\cite{Izumiya}).

A 3-parameter line congruence in $\R^4$ is nothing but a 3-parameter family of lines over a hypersurface in $\R^4$. Locally, we denote a line congruence by $\cal{C}$ = $\lbrace \br(u), \bx(u) \rbrace$, where $\br$ is a parametrization of the reference hypersurface $S$ and $\bx$  is a parametrization of a director hypersurface. A classical example appears when we consider the congruence generated by the normal lines to a regular hypersurface $S$ in $\R^4$, which is called an exact normal congruence. Here, we look at a line congruence $\cal{C}$ = $\lbrace \br(u), \bx(u) \rbrace$ as a smooth map $F_{(\br, \bx)}: U \times I \rightarrow \R^4$, given by $F_{(\br, \bx)}(u,t) = \br(u) + t\bx(u)$, where $I$ is an open interval and $U \subset \R^3$ is an open subset. 

Taking into account \cite{Izumiya}, we seek to provide a classification of the generic singularities of 3-parameter line congruences, 3-parameter normal congruences and Blaschke normal congruences in $\R^4$. As we want to use methods of singularity theory to classify congruences, in  section 2 we review some results that are useful for the next sections. In section 3, we give some basic definitions and results on 3-parameter line congruences.  In Sections 4 and 5 we use the same approach as in \cite{Izumiya} to classify generically the singularities of 3-parameter line congruences and 3-parameter normal congruences 
 in theorems (\ref{teo4.1}) and (\ref{teo5.2}),  respectively. The comparison of these two theorems shows that the generic singularities of 3-parameter line congruences are different from the generic singularities of 3-parameter normal congruences. Furthermore, we show that generically we also have singularities of corank $2$ in both cases and the proof of theorem (\ref{teo4.1}) relies on a refinement of $\cK$-orbits by $\cA$-orbits of $\cA_{e}$-codimension $1$.

In section 6, we look at the Blaschke (affine) normal congruences, i.e. congruences related to the Blaschke vector field of a non-degenerate hypersurface in $\R^4$, which is a classical equiaffine transversal vector field. Based on the theory of Lagrangian singularities, we define the family of support functions associated to the Blaschke congruence and prove that this is a Morse family of functions. We then classify the generic singularities of the Blaschke exact normal congruences and Blaschke normal congruences, providing a positive answer to the following conjecture presented in \cite{Izumiya}:
\begin{conjecture}
	Germs of generic Blaschke affine normal congruences at any
point are Lagrangian stable.
\end{conjecture}

\section{Fixing notations, definitions and some basic results}
We denote by $I \subset \R$ an open interval and $U$ an open subset of $\R^3$, where $t \in I$ and $u = (u_{1}, u_{2}, u_{3}) \in U$. Here, $\br: U \rightarrow \R^4$ is not necessarily an immersion, i.e. it may have singularities. Given any smooth map $f: U \rightarrow \R$, we denote by $f_{u_{i}}$ the derivative of $f$ with respect to $u_{i}$, $i=1,2,3$.

We now present some basic results in singularity theory which help us in the next sections. More details can be found in \cite{gib}, \cite{wall} and \cite{juanjo}. Given map germs $f,g: (\R^n, \textbf{0}) \rightarrow (\R^p, \textbf{0})$, if there is a germ of a diffeomorphism $h: (\R^n, \bm{0}) \rightarrow (\R^n, \bm{0})$, such that $h^{*}(f^{*}(\cM_{p})) = g^{*}(\cM_{p})$, where $h^{*}(f^{*}(\cM_{p}))$ is the ideal generated by the coordinate functions of $f \circ h$ and  $ g^{*}(\cM_{p})$ is the ideal generated by the coordinate functions of $g$, we say that  $f$ and $g$ are $\cK$-equivalent, denoted by, $f \isEquivTo{\cK} g $. Let $J^{k}(n,p)$ be the $k$-jet space of map germs from $\R^{n}$ to $\R^{p}$. For any $ j^{k}f(0)$, we set
\begin{align*}
\cK^{k}(j^{k}f(0)) = \lbrace j^{k}g(0): f \isEquivTo{\cK} g  \rbrace,
\end{align*}
for the $\cK$-orbit of $f$ in the space of $k$-jets $J^{k}(n,p)$. For a map germ $f: (\R^n \times \R^r, \textbf{0}) \rightarrow (\R^{p}, \textbf{0})$ we define
\begin{align*}
j^{k}_{1}: (\R^n \times \R^r, \textbf{0}) &\rightarrow J^{k}(n,p)\\
(x,u) &\mapsto j^{k}_{1}(x,u) = j^{k}f_{u}(x),
\end{align*}
where $f_{u}(x) = f(x,u)$.

The next definition of unfolding is locally equivalent to the usual parametrized one (see \cite{gibsonand}, chapter 3).
\begin{defi}\normalfont
Let $f: (N, x_{0}) \rightarrow  (P, y_{o})$ be a map germ between manifolds. An \textit{unfolding} of $f$ is a triple $(F,i,j)$ of map germs, where $i: (N, x_0) \rightarrow (N', x'_{0})$, $j:(P, y_0) \rightarrow (P', y'_{0}) $ are immersions and $j$ is transverse to $F$, such that $F \circ i = j \circ f$ and $(i,f): N \rightarrow \lbrace (x', y)\in N' \times P: F(x') = j(y) \rbrace $ is a diffeomorphism germ. The dimension of the unfolding is $dim(N') - dim(N)$.
\begin{figure}[h]
\begin{center}
		\includegraphics[scale=0.2]{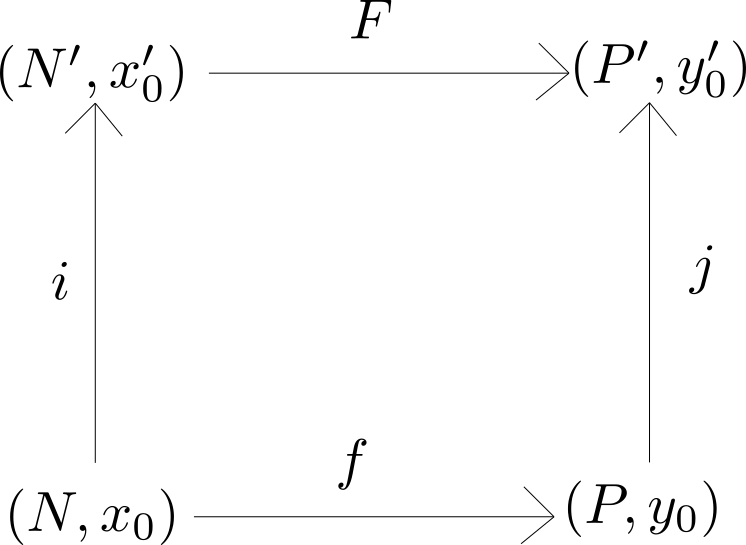}
		\caption{Associated diagram}
\end{center}
\end{figure}
\end{defi}
\begin{lema}{\rm{(\cite{Izumiya}, Lemma 3.1)}}\label{lema3.1}
Let $F: (\R^{n-1} \times \R, (\bm{0}, 0)) \rightarrow (\R^n, \bm{0})$ be a map germ with components $F_{i}(x,t)$, $i=1,2,\cdots, n$, i.e.
\begin{align*}
F(x,t) = (F_{1}(x,t), \cdots, F_{n}(x,t)).
\end{align*}
Suppose that $\dfrac{\partial F_{n}}{\partial t}(0,0) \neq 0$. We know by the Implicit Function Theorem that, there is a germ of function $g: (\R^{n-1}, \bm{0}) \rightarrow (\R, 0)$, such that
\begin{align*}
F_{n}^{-1}(0) = \lbrace (x, g(x)): x \in (\R^{n-1}, \bm{0})  \rbrace.
\end{align*}
Let us consider the immersion germs $i: (\R^{n-1}, \bm{0}) \rightarrow (\R^n, (\bm{0},0))$, given by $i(x) = (x,g(x))$, $j: (\R^{n-1}, \bm{0}) \rightarrow (\R^n, (\bm{0}, 0))$, given by $j(y) = (y,0)$ and a map germ $f: (\R^{n-1}, \bm{0}) \rightarrow (\R^{n-1}, \bm{0})$, given by $f(x) = (F_{1}(x,g(x)), \cdots, F_{n-1}(x,g(x)))$. Then the triple $(F,i,j)$ is a one-dimensional unfolding of $f$.
\end{lema}

\begin{lema}{\rm{(\cite{Izumiyaruled}, Lemma 3.3})}\label{lema3.2}
	Let $F : (\R^n \times \R^r, \bm{0}) \rightarrow (\R^p \times \R^r, \bm{0})$ be an unfolding of $f_{0}$ of the form
	$F (x, u) = (f(x, u), u)$. If $j^{k}_{1}f$ is transverse to $\cK^{k}(j^{k}f_{0}(0))$ for a sufficiently large $k$, then $F$ is infinitesimally $\cA$-stable.
\end{lema}

\begin{defi}\normalfont
	We say that a $r$-parameter family of germs of functions $F : (\R^n \times \R^r, \bm{0}) \rightarrow (\R, 0)$ is a \textit{Morse family of functions} if the map germ $\Delta_{F}:(\R^n \times \R^r, \bm{0}) \rightarrow (\R^n, \bm{0}) $, given by
	\begin{align*}
	\Delta_{F}(x,u) = \left( \frac{\partial F}{\partial x_{1}}, \cdots, \frac{\partial F}{\partial x_{n}}\right)(x,u)
	\end{align*}
	 is not singular.
\end{defi}

\begin{defi}\normalfont
	Let $\cG$ be one of Mather's subgroups of $\cK$ and $\cB$ a smooth manifold. A family of maps $F: \R^n \times \cB \rightarrow \R^k$, given by $F(x,u) = f_{u}(x) $, is said to be \textit{locally $\cG$-versal}  if for every $(x,u) \in \R^n \times \cB$, the germ of $F$ at $(x,u)$ is a $\cG$-versal unfolding of $f_{u}$ at $x$.
\end{defi}
Let $g: M \rightarrow \R^{n}$ be an immersion, where $M$ is a smooth manifold, and denote by $\phi_{g}: M \times \cB \rightarrow \R^k$ the map given by
\begin{align*}
\phi_{g}(y,u) = F(g(y), u).
\end{align*}
\begin{teo}{\rm{(\cite{Montaldi}, Theorem 1)}}\label{teoMontaldiadap}
Suppose $F: \R^n \times \cB \rightarrow \R^k$ as above is locally $\cG$-versal. Let $W \subset J^{r}(M, \R^k)$ be a $\cG$-invariant submanifold, where $M$ is a manifold and let
\begin{align*}
R_{W} = \lbrace g \in Imm(M, \R^n): j^{r}_{1}\phi_{g} \pitchfork  W \rbrace.
\end{align*}
Then $R_{W}$ is residual in $Imm(M, \R^n)$. Moreover, if $\cB$ is compact and $W$ is closed, then $R_{W}$ is open and dense.
\end{teo}

\section{Line Congruences}
In this section, we define 3-parameter line congruences and discuss some of their properties.

\begin{defi}\normalfont
	A \textit{3-parameter line congruence in $\R^4$} is a 3-parameter family of lines in $\R^4$. Locally, we write $\cal{C}$ = $\lbrace \br(u), \bx(u) \rbrace$ and the line congruence is given by a smooth map
	\begin{align*}
	F_{(\br, \bx)}: U \times I &\rightarrow \R^4\\
	(u,t) &\mapsto F(u,t) = \br(u) + t\bx(u), 
	\end{align*}
	where 
	\begin{itemize}
		\item $\br: U \rightarrow \R^4$ is smooth and it is called a \textit{reference hypersurface of the congruence};
		\item $\bx: U \rightarrow \R^4 \setminus \lbrace \textbf{0} \rbrace$ is smooth and it is called the \textit{director hypersurface of the congruence}.
	\end{itemize}
\end{defi}
When there is no risk of confusion, we denote the line congruence just by $F$ instead of $F_{(\br, \bx)}$.

\begin{lema}
	The singular points of a line congruence $F_{(\br, \bx)}$ are the points $(u,t)$ such that
	\begin{align*}
	& t^3 \langle \bx, \bx_{u_1}\wedge \bx_{u_2} \wedge \bx_{u_3}  \rangle + t^2 \langle \bx, \br_{u_1}\wedge \bx_{u_2} \wedge \bx_{u_3} +  \bx_{u_1}\wedge \br_{u_2} \wedge \bx_{u_3} + \bx_{u_1}\wedge \bx_{u_2} \wedge \br_{u_3} \rangle + \\
	&+ t \langle \bx, \br_{u_1}\wedge \br_{u_2} \wedge \bx_{u_3} +  \br_{u_1}\wedge \bx_{u_2} \wedge \br_{u_3} + \bx_{u_1}\wedge \br_{u_2} \wedge \br_{u_3} \rangle + \langle \bx, \br_{u_1}\wedge \br_{u_2} \wedge \br_{u_3}  \rangle = 0.
	\end{align*}
\end{lema}
\begin{proof}
	The jacobian matrix of $F$ is
	\begin{align*}
	JF = \left[ \br_{u_1} + t\bx_{u_1} \ \ \br_{u_2} + t\bx_{u_2} \ \ \br_{u_3} + t\bx_{u_3} \ \ \bx \right].
	\end{align*}	
	As we know, $(u,t)$ is a singular point of $F$ if, and only if, $\det JF(u,t) = 0$, thus the result follows from
	\begin{align*}
	\det JF(u,t) = \langle \bx, (\br_{u_1} + t\bx_{u_1}) \wedge (\br_{u_2} + t\bx_{u_2}) \wedge (\br_{u_3} + t\bx_{u_3}) \rangle = 0.
	\end{align*}
\end{proof}

\begin{defi}\label{def3.2}\normalfont
	We say that $\by(u) = \br(u) + t(u)\bx(u)$ is a \textit{focal hypersurface} of the line congruence $F_{(\br, \bx)}$ if
	\begin{align}
	\langle \bx(u), \by_{u_{1}} \wedge  \by_{u_{2}} \wedge\by_{u_{3}} \rangle = 0.
	\end{align}
\end{defi}

	If $\by(u) = \br(u) + t(u)\bx(u)$ is a focal hypersurface of the line congruence $F_{(\br, \bx)}$ then
	\begin{align*}
	& t^3 \langle \bx, \bx_{u_1}\wedge \bx_{u_2} \wedge \bx_{u_3}  \rangle + t^2 \langle \bx, \br_{u_1}\wedge \bx_{u_2} \wedge \bx_{u_3} +  \bx_{u_1}\wedge \br_{u_2} \wedge \bx_{u_3} + \bx_{u_1}\wedge \bx_{u_2} \wedge \br_{u_3} \rangle + \\
	&+ t \langle \bx, \br_{u_1}\wedge \br_{u_2} \wedge \bx_{u_3} +  \br_{u_1}\wedge \bx_{u_2} \wedge \br_{u_3} + \bx_{u_1}\wedge \br_{u_2} \wedge \br_{u_3} \rangle + \langle \bx, \br_{u_1}\wedge \br_{u_2} \wedge \br_{u_3}  \rangle = 0.
	\end{align*}

\subsection{Ruled surfaces of the congruence}
	There is a geometric interpretation related to definition (\ref{def3.2}), when $\br$ is an embedding and $\bx$ is an immersion, as follows. Let $\lbrace \br(u), \bx(u) \rbrace$ be a 3-parameter line congruence and $C$ a regular curve on the reference hypersurface $\br$. If we restrict the director hypersurface $\bx$ to this curve, we obtain a ruled surface associated to the 1-parameter family of lines $\lbrace \br(s), \bx(s) \rbrace$, where $s$ is the parameter of $C$, $\br(s) = \br(u(s))$ and $\bx(s) = \bx(u(s))$. The line obtained by fixing $s$ is called a \textit{generator of the ruled surface}. These kind of ruled surfaces are called \textit{surfaces of the congruence} and since $\bx'(s) \neq 0$, it is possible to define its striction curve (see section 3.5 in \cite{manfredo} for details). In the special case where this ruled surface is developable, the points of contact of a generator with the striction curve are called  \textit{focal points}. Let us write  $\alpha(s) = \br(u(s)) + \rho(u(s))\bx(u(s)) $ as the striction curve, where $\rho(u(s))$ denotes the coordinate of the focal point relative to $\bx(u(s))$. Suppose $\alpha'(s) \neq 0$ for all $s$, then it is possible to show that $\alpha'$ is parallel to $\bx$ and assuming $\lVert \bx \rVert = 1$, $\alpha'$ is perpendicular to $\bx_{u_{i}}$, $i=1,2,3$, thus
	\begin{align*}
	\begin{cases}
	u_{1}'(h_{11} + \rho g_{11}) + u_{2}'(h_{21} + \rho g_{12}) + u_{3}'(h_{31} + \rho g_{13}) = 0\\ 
	u_{1}'(h_{12} + \rho g_{12}) + u_{2}'(h_{22} + \rho g_{22}) + u_{3}'(h_{32} + \rho g_{23}) =0\\
	u_{1}'(h_{13} + \rho g_{13}) + u_{2}'(h_{23} + \rho g_{23}) + u_{3}'(h_{33} + \rho g_{33}) = 0,
	\end{cases}
	\end{align*}
	where $g_{ij} = \langle \bx_{u_{i}}, \bx_{u_{j}} \rangle$ and $h_{ij} = \langle \br_{u_{i}}, \bx_{u_{j}} \rangle$. As we want to find a non-trivial solution for the above system, we obtain the cubic equation
	\begin{align*}
	\left| \begin{array}{ccc}
	h_{11} + \rho g_{11} & h_{21} + \rho g_{12} & h_{31} + \rho g_{13} \\ 
	h_{12} + \rho g_{12} & h_{22} + \rho g_{22} &  h_{32} + \rho g_{23}\\
	h_{13} + \rho g_{13} & h_{23} + \rho g_{23} & h_{33} + \rho g_{33}
	\end{array} \right| = 0,
	\end{align*}
	from which we obtain the coordinates $\rho_{i}$ of the focal points, $i=1,2,3$. Hence, related to each line of the congruence we have (possibly) three focal points. We define a \textit{focal set of the congruence} as
	\begin{center}
		$\bm{y}_{i}(u) = \br(u) + \rho_{i}(u)\bx(u),\; i=1,2,3 $.
	\end{center}
	Thus, for every $u_{0}$, $\bm{y}_{i}(u_{0})$ is a focal point and there is a curve in this focal set (striction curve) $\alpha(s) = \br(u(s)) + \rho_{i}(u(s))\bx(u(s)) $, such that $\alpha(s_{0}) = \bm{y}_{i}(u_{0})$ and $\alpha'(s_{0})$ is parallel to $\bx(u_{0})$, then	\begin{align}
	\langle \bx(u_{0}), \by_{iu_{1}} \wedge  \by_{iu_{2}} \wedge\by_{iu_{3}} \rangle = 0.
	\end{align}
	Therefore, the focal points are located at the focal hypersurfaces defined in (\ref{def3.2}).

\section{Generic classification of 3-parameter line congruences in $\R^4$}
In this section we use methods of singularity theory to obtain the generic singularities  of 3-parameter line congruences in $\mathbb{R}^4$. Our approach is the same as in \cite{Izumiya}, but here we are dealing with the case of 3 parameters in $\R^4$. Let $F_{(\br, \bx)}$ be a line congruence and take $x_{i}$ and $\xi_{i}$, $i=1,2,3,4$, as the coordinate functions of $\br$ and $\bx$, respectively, thus we have 
\begin{align*}
F_{(\br, \bx)}(u,t) = \left( x_{1}(u) + t\xi_{1}(u), x_{2}(u) + t\xi_{2}(u), x_{3}(u) + t\xi_{3}(u), x_{4}(u) + t\xi_{4}(u)  \right).
\end{align*}
If $(u_{0}, t_{0}) \in U \times I$ and $\xi_{4}(u_{0}) \neq 0$ then there exists $U_{4} \subset U$ an open subset given by $\lbrace u \in U: \xi_{4}(u) \neq 0  \rbrace$. Let us define
\begin{align}\label{eqa0}
c_{4}(u) = -\dfrac{x_{4}(u) - a_{0}}{\xi_{4}(u)},
\end{align}
where $u \in U_{4}$ and $a_{0} = x_{4}(u_{0}) + t_{0}\xi_{4}(u_{0})$. Therefore,
\begin{align*}
F_{(\br,\bx)}(u,t) &= \br(u)  + c_{4}(u)\bx(u) + \left( t - c_{4}(u) \right)\bx(u)\\
&= \br(u) + c_{4}(u)\bx(u) + \tilde{t}\bx(u),\; where\; \tilde{t} = t - c_{4}(u).
\end{align*}
Then, if we look at $\widetilde{F}_{(\br, \bx)}(u,\tilde{t}) = \br(u) + c_{4}(u)\bx(u) + \tilde{t}\bx(u)$ we can see that its fourth coordinate, which is denoted by $\widetilde{F}_{4}$, is $x_{4}(u) + c_{4}(u)\xi_{4}(u) + \tilde{t}\xi_{4}(u) =a_{0} + \tilde{t}\xi_{4}(u)$, by (\ref{eqa0}). Furthermore, $\widetilde{F}_{4}^{-1}(a_{0}) = \lbrace (u,0): u \in U_{4} \rbrace$ and via the Implicit Function Theorem and lemma (\ref{lema3.1}), the germ of $\widetilde{F}_{(\br, \bx)}$ at $(u_{0},0)$ is an one-dimensional unfolding of 
\begin{align*}
\tilde{f}(u) = \tilde{\pi}_{4} \circ \widetilde{F}_{(\br, \bx)}(u,0) = \left( x_{1}(u) + c_{4}(u)\xi_{1}(u), x_{2}(u) + c_{4}(u)\xi_{2}(u), x_{3}(u) + c_{4}(u)\xi_{3}(u) \right),
\end{align*}
where $\tilde{\pi}_{4}(y_{1}, y_{2}, y_{3}, y_{4}) = (y_{1}, y_{2}, y_{3}) $.

\begin{lema}
Let $F_{(\br, \bx)}: U \times I \rightarrow \R^4$ be a line congruence. With notation as above, the singularity of $\tilde{f}$ at $u_{0}$ is determined by $\tilde{\pi}_{4} \circ \br$.
\end{lema}

\begin{proof}
	Let us suppose $\bx_{4}(u_{0}) \neq 0$ (other cases are analogous),  $(u_{0}, t_{0}) = (0,0) \in U \times I$ and $\bx(0) = (0,0,0,1)$. Using the above notation, $c_{4}(0) = 0$, thus the jacobian matrix of $\tilde{f}$ at $0$ is equal to the jacobian matrix of $\tilde{\pi}_{4} \circ \br$ at $0$.
\end{proof}

The above lemma is important because it shows that the singularity of $\tilde{f}$, and therefore the unfolding $\widetilde{F}$, is determined by $\tilde{\pi}_{4} \circ \bx: U \rightarrow \R^3$.

\begin{lema}{\rm{({\cite{gg}, Lemma 4.6})}}\rm{(Basic Transversality Lemma)}
		Let $X$, $B$ and $Y$ be smooth manifolds with $W$ a submanifold of $Y$. Consider $j: B \rightarrow C^{\infty}\left(X, Y \right)$ a non-necessarily continuous map and define $\Phi: X \times B \rightarrow Y$ by $\Phi(x,b) = j(b)(x)$. Suppose $\Phi$ smooth and transversal to $W$, then the set
	\begin{align*}
	\left\lbrace b \in B: j(b) \pitchfork W \right\rbrace
	\end{align*}
	is a dense subset of $B$.
\end{lema}
 The next lemma is the result for 3-parameter line congruences in $\R^4$ which corresponds to the lemma 4.1 in \cite{Izumiya}.
\begin{lema}\label{Lema4.1}
Let $W \subset J^{k}(3,3)$ be a submanifold. For any fixed map germ  $\bx: U \rightarrow \R^{4} \setminus \lbrace 0 \rbrace$ and any fixed point $(u_{0}, t_{0}) \in U \times I$ with $\xi_{4}(u_{0}) \neq 0$, the set 
\begin{align*}
T^{\bx}_{4,W, (u_{0}, t_{0})} = \left\lbrace \br \in C^{\infty}(U, \R^4): j^{k}_{1} \left( \tilde{\pi}_{4} \circ \widetilde{F}_{(\br,\bx)} \right) \pitchfork W \ \ at \ \ (u_{0}, t_{0})  \right\rbrace
\end{align*}
is a residual subset of $C^{\infty}\left( U, \R^{4} \right)$.
\end{lema}

\begin{proof}
	See lemma 4.1 in \cite{Izumiya}.	
	\end{proof}

If $\xi_{j}(u_{0}) \neq 0$, $j=1,2,3$, we can define the set
\begin{align*}
T^{\bx}_{j,W, (u_{0}, t_{0})} = \left\lbrace \br \in C^{\infty}(U, \R^4): j^{k}_{1} \left( \tilde{\pi}_{j} \circ \widetilde{F}_{(\br,\bx)} \right) \pitchfork W \ \ at \ \ (u_{0}, t_{0})  \right\rbrace,\; j=1,2,3
\end{align*}
where $\tilde{\pi}_{j}$ is the projection in the coordinates different than $j$. Thus, the above lemma holds for the sets $T^{\bx}_{j,W, (u_{0}, t_{0})}$, $j=1,2,3,4$.

\begin{obs}\label{obsconj}\normalfont
Define
\begin{align*}
{\cal{O}}_{1} =& \left\{ \bx \in C^{\infty}\left(U, \R^{4} \setminus \lbrace \textbf{0} \rbrace \right): \bx_{u_1} \wedge \bx_{u_2} \wedge \bx \neq \textbf{0},\; or\; \bx_{u_1} \wedge \bx_{u_3} \wedge \bx \neq \textbf{0},\right. \\     & \left. or\;   \bx_{u_2} \wedge \bx_{u_3} \wedge \bx \neq \textbf{0} , \forall\; u \in U \right\}
\end{align*}
Then, ${\cal{O}}_{1}$ is residual as follows. Take the matrix $\left[ \bx \ \ \bx_{u_1} \ \ \bx_{u_2} \ \ \bx_{u_3} \right]$ and suppose that $\bx \notin {\cal{O}}_{1} $. Thus, $ \bx_{u_i}(u_{0}) \wedge \bx_{u_j}(u_{0}) \wedge \bx(u_{0}) = 0$, for some $u_{0} \in U$ and $i,j=1,2,3$, i.e., the sets $\lbrace \bx(u_{0}), \bx_{u_i}(u_{0}), \bx_{u_j}(u_{0}) \rbrace$ are linearly dependent. So we have two cases: 
\begin{enumerate}
	\item  $\lbrace \bx_{u_1}(u_{0}), \bx_{u_2}(u_{0}), \bx_{u_3}(u_{0}) \rbrace$ is a LI set, then we would have $\bx(u_{0})=0$. Contradiction.
	\item $\lbrace \bx_{u_1}(u_{0}), \bx_{u_2}(u_{0}), \bx_{u_3}(u_{0}) \rbrace$ is a LD set, thus, $\left[ \bx(u_{0}) \; \bx_{u_1}(u_{0})\; \bx_{u_2}(u_{0})\; \bx_{u_3}(u_{0}) \right]$ has rank less than or equal to $2$. 
\end{enumerate} 
Let $\Sigma$ be a submanifold of the space of $4 \times 4$ matrices formed by the matrices with $rank \leq 2$, i.e., in which the minors of order $3 \times 3$ are zero, so, $\Sigma$ has codimension $4$. Since $\bx \in C^{\infty}\left(U, \R^{4} \setminus \lbrace \textbf{0} \rbrace \right)$ is such that $j^{1}\bx \pitchfork \Sigma$ and $U$ is an open subset of $\R^3$, we have $j^{1}\bx(U) \cap \Sigma = \emptyset$, what happens if, and only if, $\bx \in {\cal{O}}_{1}$. Therefore, by Thom's Transversality Theorem, ${\cal{O}}_{1}$ is residual. Note above that we are denoting $j^{1}\bx(u) = [\bx(u) \ \ \bx_{u_1}(u) \ \ \bx_{u_2}(u) \ \ \bx_{u_3}(u)]$.

Thus, it follows from lemma (\ref{Lema4.1}) that
\begin{align*}
\widetilde{T}_{4,W, (u_{0}, t_{0})} &= \left\lbrace (\br, \bx): j^{k}_{1} \left( \tilde{\pi}_{4} \circ \widetilde{F}_{(\br, \bx)} \right) \pitchfork W\; at \; (u_{0}, t_{0}), \bx \in  {\cal{O}}_{1}  \right\rbrace
\end{align*}
is residual.
\end{obs}

Now, we are able to prove our first main theorem, which provides a classification of the generic singularities of 3-parameter line congruences in $\R^4$.
\begin{teo}\label{teo4.1}
	There is an open dense set ${\cal{O}} \subset C^{\infty} \left( U, \mathbb{R}^4 \times (\R^{4} \setminus \lbrace  0 \rbrace) \right)$, such that:
	
	\begin{enumerate}[(a)]
		\item  For all $(\br, \bx) \in \cal{O}$, the germ of the line congruence $F_{(\br, \bx)}$ at any point $(u_{0}, t_{0}) \in U \times I$ is stable;
		\item For all $(\br, \bx) \in \cal{O}$, the germ of the line congruence $F_{(\br, \bx)}$ at any point $(u_{0}, t_{0}) \in U \times I$ is a 1-parameter versal unfolding of a germ $f: (\R^3,  u_{0}) \rightarrow \R^3$ at $t = t_{0}$. Then, $F_{(\br, \bx)}$ is $\cal{A}$-equivalent to one of the normal forms below
	\end{enumerate}
	
	\begin{itemize}
		\item $(x,y,z,w) \mapsto (x,y,w, z^2)$ {\rm{(Fold)}}.
		\item $(x,y,z,w) \mapsto (x,y,w, z^3 + xz)$ {\rm{(Cusp)}}.
		\item $(x,y,z,w) \mapsto (x,y,z^3 + (x^2 \pm y^2)z + wz, w)$ {\rm{(Lips/Beaks)}}.
		\item $(x,y,z,w) \mapsto (x,y,w,z^4 + xz + yz^2)$ {\rm{(Swallowtail)}}.
		\item $(x,y,z,w) \mapsto (x,y,w,z^4 + xz \pm y^2z^2 + wz^2)$.
		\item $(x,y,z,w) \mapsto (x,y,w,z^5 + xz + yz^2 + wz^3)$ {\rm{(Butterfly)}}.
		\item $\left( x,y,z,w \right) \mapsto (z, x^2 + y^2 + zx + wy, xy, w) $ {\rm{(Hyperbolic Umbilic)}}.
		\item $\left( x,y,z,w \right) \mapsto (z, x^2 - y^2 + zx + wy, xy, w) $ {\rm{(Elliptic Umbilic)}}.		
	\end{itemize}
\end{teo}

\begin{proof} We first prove item $(a)$.
	Given $f \in \cE_{3,3}$ and $z = j^{k}f(0)$, define
	\begin{align*}
	{\cal{K}}^{k}(z) = \lbrace j^kg(0): g \isEquivTo{\K} f \rbrace.
	\end{align*}
	For a sufficiently large $k$, define $$\Pi_{k}(3,3) = \lbrace f \in J^{k}(3,3):\, cod_{e}({\cal{K}}, f) \geq 5  \rbrace.$$ Consider $$\Sigma^{i} = \lbrace \sigma \in J^{1}(3,3): \text{kernel rank}(\sigma) = i \rbrace \subset J^{1}(3,3),$$
	which is a submanifold of  codimension $i^2$.
	
	\begin{enumerate}
		\item 	We look at the slice of $\Pi_{k}(3,3)$ in $\Sigma^{1}$, i.e.,  $f \in \Pi_{k}(3,3)$ such that $kernel\; rank(df(0)) = 1$. Then, we are dealing with $f \in \Pi_{k}(3,3)$ of $corank$ $1$. Therefore, we can write $f(x,y,z) = (x,y, g(x,y,z))$, where $g(0,0,z)$ has a singularity of $A_{r}$ type, for some $5 \leq r \leq k-1$ and we call them ${\cal{K}}$-singularities of $A_{r}$-type. Note that if we regard the ``good" set as the complement of $\Pi_{k}(3,3)$ in $\Sigma^{1}$, then its singularities are the ${\cK}$-singularities of $A_{1}$, $A_{2}$, $A_{3}$ and $A_{4}$-type. Therefore, the slice $\Pi_{k}(3,3) \cap \Sigma^{1} $ is a semialgebraic set of codimension greater than or equal to $5$, so it has a stratification $\lbrace {\cal{S}}_{i}^{1} \rbrace_{i=1}^{m_{1}}$, with $codim({\cal{S}}_{i}^{1})\geq 5$.

	\item As we did in the first case, define $\Pi_{k}(3,3) \cap \Sigma^{2} $, i.e., the set of $f \in \Pi_{k}(3,3)$ of $corank$ $2$.  We may assume that $f(x,y,z) = (z, g_{1}(x,y,z), g_{2}(x,y,z)) $, where $g_{i}$ has zero $1$-jet and $(g_{1}(x,y,0), g_{2}(x,y,0))$ has $2$-jet in $H^{2}(2,2)$, therefore, $(g_{1}(x,y,0), g_{2}(x,y,0))$ has $2$-jet given by one of the normal forms below (See \cite{gib} or \cite{juanjo}):
	\begin{align*}
	(x^2 + y^2, xy);\; (x^2-y^2, xy);\; (x^2, xy);\; (x^2, 0);\; (x^2 \pm y^2, 0);\; (0,0).
	\end{align*}
	Hence, by looking at the first two normal forms and its local algebras, $f$ is $\cK$-equivalent to one of the forms below:
	\begin{itemize}
		\item $W_{1}: \left(z, x^2 + y^2 + xz, xy \right)$
		\item $W_{2}: \left(z, x^2 - y^2 + xz, xy \right)$
		\end{itemize}
	and both of these forms have $cod_{e}({\cal{K}}) = 4$. The other ${\cal{K}}$-orbits have $cod_{e}({\cal{K}})\geq5$.
	Note that $\overline{\Sigma^{2}} \setminus \left( W_{1} \cup W_{2} \right) $ is a semialgebraic set of codimension greater than or equal to $5$.  $\Pi_{k}(3,3) \cap \Sigma^{2}$ is a semialgebraic set contained in $ \overline{\Sigma^{2}} \setminus \left( W_{1} \cup W_{2} \right)$, then its codimension is greater than or equal to $5$, thus, there is a stratification $\lbrace {\cal{S}}_{i}^{2} \rbrace_{i=1}^{m_{2}}$ of it, with $codim({\cal{S}}_{i}^{2})\geq 5$. 
		Furthermore, the ``good" set contains only   $W_{1}$ and $W_{2}$.

		\item In a similar way, we define $\Pi_{k}(3,3) \cap \Sigma^{3} $, i.e., the set of the $k$-jets $f \in \Pi_{k}(3,3)$ whose $corank$ is $3$. It is well-known that $\Sigma^{3}$ has codimension $9$, so $\Pi_{k}(3,3) \cap \Sigma^{3} $ is a semialgebraic set of codimension greater than or equal to $9$, hence, there is a stratification  $\lbrace {\cal{S}}_{i}^{3} \rbrace_{i=1}^{m_{3}}$, with $codim({\cal{S}}_{i}^{3})> 5$.

	Then, it follows that the ``good" set, i.e., the set of the ${\cal{K}}$-orbits of codimension less than or equal to $4$, contains the following ${\cal{K}}$-orbits
	\begin{itemize}
		\item type $A_{r}$, for $1 \leq r \leq 4$;
		\item type $W_{1}$;
		\item type $W_{2}$.
	\end{itemize}
	
\end{enumerate}

Applying lemma (\ref{Lema4.1}) and remark (\ref{obsconj}) to each strata of the above stratification, we obtain that
\begin{align*}
\cT_{j} = \bigcap\limits_{i=1}^{m_{j}}\widetilde{T}_{4,S_{i}^{j}, (u_{0}, t_{0})},\; j=1,2,3\\
\cT_{3+r} = \widetilde{T}_{4,A_{r}, (u_{0}, t_{0})}, \, 1\leq r \leq 4 \\
\cT_{7+i} = \widetilde{T}_{4,W_{i}, (u_{0}, t_{0})}, \, i=1,2.
\end{align*} are residual subsets of $C^{\infty}(U, \R^{4} \times (\R^{4} \setminus \lbrace 0 \rbrace))$. Hence,
\begin{align*}
	{\cal{O}}_{4,(u_{0}, t_{0})} =  \bigcap\limits_{i=1}^{9}\cT_{i} 
	\end{align*}
is residual. The same is true for the sets ${\cal{O}}_{j,(u_{0}, t_{0})}$, $j=1,2,3$, defined in a similar way.

Since $\bx(u) \neq 0$ for all $u \in U$, given a point $(u_{0}, t_{0}) \in U \times I$, $\xi_{j}(u_{0}) \neq 0$, for some $j$, there is a residual set ${\cal{O}}_{(u_{0}, t_{0})} \subset C^{\infty}(U, \R^{4} \times (\R^{4} \setminus \lbrace 0 \rbrace))$, such that 
\begin{align*}
(\br,\bx) \in {\cal{O}}_{(u_{0}, t_{0})} \Leftrightarrow j^{k}_{1}\left( \tilde{\pi}_{j} \circ \widetilde{F}_{(\br,\bx)} \right) \pitchfork {\cal{A}}_{r},\, W_{1}, \, W_{2}, \, {\cal{S}}^{j}_{i},\, j=1,2,3, \, r=1, \cdots, 4.
\end{align*}
It follows from what we already have done that the germ of $\widetilde{F}_{(\br,\bx)}$ at $(u_{0}, 0)$, which is equivalent to the germ of $F_{(\br,\bx)}$ at $(u_{0}, t_{0})$, is a 1-dimensional unfolding of $\tilde{\pi}_{j} \circ \widetilde{F}(u,0)$ and it follows from lemma (\ref{lema3.2}) that $F_{(\br,\bx)}$ is ${\cal{A}}$-infinitesimally stable for all $(\br,\bx) \in {\cal{O}}_{(u_{0}, t_{0})}$. Since a germ  ${\cal{A}}$-infinitesimally stable is ${\cal{A}}$-stable (see \cite{Mather}), there is a neighborhood $U_{u_{0}} \times I_{t_{0}}$ of $(u_{0}, t_{0})$ in $U \times I$, such that $F_{(\br,\bx)}{\big|}_{U_{u_{0}} \times I_{t_{0}}}$ is ${\cal{A}}$-stable. This result holds independently of the fixed point $(u_{0}, t_{0})$, so we can consider a countable family of points $(u_{i}, t_{i}) \in U \times I$ and neighborhoods $U_{u_{i}} \times I_{t_{i}}$, $(i=1,2, \cdots)$, such that $F_{(\br,\bx)}{\big|}_{U_{u_{i}} \times I_{t_{i}}}$ is ${\cal{A}}$-stable and 
\begin{align*}
U \times I = \bigcup\limits_{i=1}^{\infty} U_{u_{i}} \times I_{t_{i}}.
\end{align*}
Since ${\cal{O}}_{(u_{i}, t_{i})}$ is a residual subset of $ C^{\infty}(U, \R^{4} \times (\R^{4} \setminus \lbrace 0 \rbrace))$, it follows that
\begin{align*}
{\cal{O}}_{2} = \bigcap\limits_{i=1}^{\infty} {\cal{O}}_{(u_{i}, t_{i})}
\end{align*}
is residual. Furthermore, the germ of $F_{(\br,\bx)}$ at any point $(u,t) \in U \times I$ is ${\cal{A}}$-infinitesimally stable, for all $(\br,\bx) \in {\cal{O}}_{2} $.

Since ${\cal{F}}: C^{\infty}(U, \R^{4} \times (\R^{4} \setminus \lbrace 0 \rbrace)) \rightarrow C^{\infty}(U \times I, \R^{4})$, defined by ${\cal{F}}(\br,\bx)= F_{(\br,\bx)}$, is continuous and
\begin{align*}
S=\lbrace f \in  C^{\infty}(U \times I, \R^{4}): f \ \ {\cal{A}}-infinitesimally \ \ stable \rbrace
\end{align*}
is open (See \cite{gg} p. 111), ${\cal{O}} = {\cal{F}}^{-1}(S) $ is open. By previous arguments ${\cal{O}}_{2}\subset {\cal{O}}$ and ${\cal{O}}_{2}$ is dense, therefore ${\cal{O}}$ is an open dense subset.

To prove $(b)$, we refine the $\cK$-orbits of type $A_2$ and $A_3$ of the above stratification, by taking the $\cA$-orbits of $\cA_e$-codimension $\leq 1$ inside these $\cK$-orbits. Then, the relevant strata in this stratification are the $\cA$-orbits of stable singularities $A_k$, $k=1, 2, 3$, and the $\cA$-orbits of singularities of $\cA_e$-codimension $1$ of type $A_2$, $A_3$, $A_{4}$ and $D_4$. The complement of their union is a semialgebraic set of codimension greater than or equal to $5$.

\begin{enumerate}
\item ${\cal{K}}$-orbit of $A_{1}$ type 

$f(x,y,z) = (x,y,z^2)$ which is stable, hence, we have just this  ${\cal{A}}$-orbit. Its suspension in $\R^4$ is the stable germ that we are looking for.

\item ${\cal{K}}$-orbits of $A_{2}$ type 

It follows from the classification made by Marar and Tari \cite{Marar}, that the possible normal forms are
\begin{align*}
f(x,y,z) = \left( x, y, z^3 + P(x,y)z \right),
\end{align*}
where $P(x,y)$ is one of the singularities $A_{k}$, $D_{k}$, $E_{6}$, $E_{7}$ or $E_{8}$ and $\text{cod}_{\text{e}}\left( {\cal{A}}, f \right) = \mu(P)$.

As we are looking for $f$ which have a versal unfolding of dimension $1$ that is a stable germ, we must have $P(x,y) = x$ or $P(x,y) = x^2 \pm y^2$. Therefore, we have the ${\cal{A}}$-orbits
\begin{align*}
f_{1}(x,y,z) &= (x,y,z^3 + xz) \ \ (\text{Cusp});\\
f_{2}(x,y,z) &= (x,y,z^3 + (x^2 \pm y^2)z) \ \  (\text{Lips} (+) \; / \; \text{Beaks} (-)),
\end{align*}
with $\text{cod}_{\text{e}}\left( {\cal{A}}, f_{1} \right) = 0$ and $\text{cod}_{\text{e}}\left( {\cal{A}}, f_{2} \right) = 1$.
The stable germs $\R^{4},0 \rightarrow \R^4,0$ are, respectively
\begin{align*}
F_{1}(x,y,z,w) &= (x,y,z^3 + xz, w);\\
F_{2}(x,y,z, w) &= (x,y,z^3 + (x^2 \pm y^2)z + wz, w).
\end{align*}
These germs are $\cA$-equivalent, however they are considered separately, because they are versal unfoldings of $f_{1}$ and $f_{2}$, respectively, which are not $\cA$-equivalent.

\item ${\cal{K}}$-orbits of  $A_{3}$ type

In a similar way, the possible normal forms are (see \cite{Marar}, section 1)
\begin{align*}
(x,y, z^4 &+ xz \pm y^kz^2), k \geq 1. \ \  \text{cod}_{\text{e}}\left( {\cal{A}}\right) = k-1; \\
(x,y, z^4 &+ (y^2 \pm x^k)z + xz^2),  k \geq 2. \ \ \text{cod}_{\text{e}}\left( {\cal{A}} \right) = k.		
\end{align*}
Hence,the useful cases are those where $k=1$ or $k=2$ in the first type of orbit, i.e.,
\begin{align*}
f_{1}(x,y,z) &= (x,y, z^4 + xz + yz^2) \ \ \text{(Swallowtail)};\\
f_{2}(x,y,z) &= (x,y, z^4 + xz \pm y^2z^2),
\end{align*}
with $\text{cod}_{\text{e}}\left( {\cal{A}}, f_{1} \right) = 0$ e $\text{cod}_{\text{e}}\left( {\cal{A}}, f_{2} \right) = 1$.
The stable germs $\R^{4},0 \rightarrow \R^4,0$ are, respectively
\begin{align*}
F_{1}(x,y,z,w)&= (x,y, z^4 + xz + yz^2, w)\\
F_{2}(x,y,z,w) &= (x,y, z^4 + xz \pm y^2z^2 + wz^2, w).
\end{align*}

\item ${\cal{K}}$-orbits of $A_{4}$ type 

Via \cite{Marar}, the possible normal forms are
\begin{align*}
(x,y, z^5 &+ xz + yz^2), \ \ \text{cod}_{\text{e}}({\cal{A}}) = 1;\\
(x,y, z^5 &+ xz + y^2z^2 + yz^3), \ \ \text{cod}_{\text{e}}({\cal{A}}) = 2;\\
(x,y, z^5 &+ xz + yz^3), \ \ \text{cod}_{\text{e}}({\cal{A}}) = 3.\\	
\end{align*}
Thus, the only case to be considered is 
\begin{align*}
f(x,y,z) = (x,y, z^5 &+ xz + yz^2),
\end{align*}
whose associated stable germ is
\begin{align*}
F(x,y,z,w) = (x,y, z^5 &+ xz + yz^2 + wz^3, w).
\end{align*}

	\item ${\cal{K}}$-orbits $W_{1}$ and $W_{2}$

The germs
\begin{align*}
F_{1}(x,y,z,w) &= (z, x^2 + y^2 + zx + wy, xy, w);\\
F_{2}(x,y,z,w) &= (z, x^2 - y^2 + zx + wy, xy, w).
\end{align*}
are, respectively, $1$-parameter versal unfoldings of (see \cite{Bruce}, section 3)
\begin{align*}
f_{1}(x,y,z) &= (z, x^2 + y^2 + zx, xy);\\
f_{2}(x,y,z) &=   \left(z, x^2 - y^2 + zx, xy \right),
\end{align*}
where $f_{1}$ and $f_{2}$ are of the type $W_{1}$ e $W_{2}$, respectively and both have $\text{cod}_{\text{e}}(\cA)=1$. Then, we conclude the proof.
\end{enumerate}

\end{proof}

\section{Normal congruences}
In this section, our approach is the same as in \cite{Izumiya} and we seek to provide a classification of the generic singularities of 3-parameter normal congruences in $\mathbb{R}^4$. For this, it is necessary to characterize normal congruences and consider some aspects of Lagrangian singularities.
\begin{defi}\normalfont
	A 3-parameter line congruence $\cal{C}$ = $\lbrace \br(u), \bx(u) \rbrace$, for $u \in U \subset \R^3$, is said to be \textit{normal} if for each point $u_{0} \in U$ there is a neighborhood $\tilde{U}$ of $u_{0}$ and a regular hypersurface, given by $\by(u) = \br(u) + t(u)\bx(u)$, whose normal vectors are parallel to $\bx(u)$, for all $u \in \tilde{U}$. The congruence is an \textit{exact normal} congruence if $\bx(u)$ is a normal vector at $\br(u)$, for all $u \in U$.
\end{defi}

The next proposition characterizes 3-parameter normal line congruences in $\R^4$ and corresponds to the Proposition 5.1 in \cite{Izumiya}.

\begin{prop}\label{propnormal}
		Let $\cal{C}$ = $\lbrace \br(u), \bx(u) \rbrace$, $u\in U \subset \R^3$, be a 3-parameter line congruence in $\R^4$. $\cal{C}$ is normal if, and only if, $h_{ij}(u) = h_{ji}(u)$, $i,j \in \lbrace 1,2,3 \rbrace$, for all   $u \in U$, where $h_{ij} = \left\langle \br_{u_i}, \left(\dfrac{\bx}{\lVert \bx \rVert}\right)_{u_j} \right\rangle$.
\end{prop}

\begin{proof}	
		Let $\cal{C}$ be a normal congruence and $S'$ a hypersurface parameterized locally by $\by(u) = \br(u) + t(u)\bx(u)$, whose normal vectors are parallel to $\bx(u)$. Let us suppose that $\lVert \bx(u) \rVert = 1$. Then, $y_{u_i}(u)$, $i=1,2,3$ are orthogonal to $\bx(u)$, therefore, $\langle \bx, \by_{u_i} \rangle = 0$. From these expressions, we obtain
	\begin{align}\label{eqdif}
	t_{u_i} = -\langle \br_{u_i}, \bx \rangle,\; i=1,2,3.
	\end{align}
	Since $t$ is smooth, $t_{u_1u_2} = t_{u_2u_1}$, $t_{u_1u_3} = t_{u_3u_1}$ and $t_{u_2u_3} = t_{u_3u_2}$. From $t_{u_1u_2} = t_{u_2u_1}$, we obtain
	\begin{align*}
	- \langle \br_{u_1u_2}, \bx  \rangle - \langle \br_{u_1}, \bx_{u_2} \rangle = - \langle \br_{u_1u_2}, \bx  \rangle - \langle \br_{u_2}, \bx_{u_1} \rangle
	\end{align*}
	Therefore, $h_{12} = \langle \br_{u_1}, \bx_{u_2} \rangle = \langle \br_{u_2}, \bx_{u_1} \rangle = h_{21}$. The other cases are analogous.

	Reciprocally, suppose $h_{ij} = h_{ji}$, for $i,j=1,2,3$. Taking into account the differential equations in (\ref{eqdif}),
	it follows from $h_{ij} = h_{ji}$ that $t_{u_1u_2} = t_{u_2u_1}$, $t_{u_1u_3} = t_{u_3u_1}$ and $t_{u_2u_3} = t_{u_3u_2}$. Therefore, the above system has a solution $t$. Write $\by(u) = \br(u) + t(u)\bx(u)$. Note that
	\begin{align*}
	\langle \bx, \by_{u_i} \rangle & = \langle \bx, \br_{u_i} \rangle + t_{u_i}\\
	&= \langle \bx, \br_{u_i} \rangle - \langle \bx, \br_{u_i} \rangle = 0.
	\end{align*}
	If $\by$ is not an immersion, there is a positive real number $\lambda$ such that $\tilde{\by}(u) = \br(u) + (t(u) + \lambda)\bx(u)$ is an immersion. For the last part, it is sufficient to look at the case when $\bm{y}(u)$ belongs to the focal set of the congruence.
	\end{proof}

Denote by
\begin{align*}
Emb(U, \R^4) = \lbrace  \br: U \rightarrow \R^4: \br\; \text{is an embedding}  \rbrace 
\end{align*}
the space of the regular hypersurfaces in $\R^4$ with the Whitney  $C^{\infty}$-topology, and by
\begin{align*}
EN\left(U, \R^4 \times \left(\R^4 \setminus \lbrace \textbf{0} \rbrace \right)\right) = \left\lbrace (\br,\bx): \br \in Emb(U, \R^4),\; \bx(u)\; \text{is normal to}\; \br\; \text{at}\; \br(u) \right\rbrace
\end{align*}
the space of the exact normal congruences. So, we have the following well known theorem.

\begin{teo}\label{teo5.1}
	There is an open dense subset $O \subset Emb(U, \R^4)$, such that the germ of an exact normal congruence $F_{(\br,\bx)}$ at any point $(u_{0}, t_{0}) \in U \times I$ is a
Lagrangian stable map germ for any $\br \in O$, i.e., $ \forall\, \br \in O$, $F_{(\br,\bx)}$ is an immersive germ, or $\cA$-equivalent to one of the normal forms in table (\ref{table1}).
		\begin{table}[!ht]
		\begin{tabular}{|l|l|}
			\hline
			\textbf{Singularity} & \textbf{Normal form}  \\ 	\hline
			Fold                                     &   $(x,y,w, z^2)$           \\	\hline
			Cusp                                   &   $(x,y,w, z^3 + xz)$           \\ 	\hline
			Swallowtail                               &  $(x,y,w,z^4 + xz + yz^2)$            \\ 	\hline
			Butterfly                             &  $(x,y,w,z^5 + xz + yz^2 + wz^3)$            \\ 	\hline
			Elliptic Umbilic                         &      $(z,w, x^2 - y^2 + zx + wy, xy)$        \\ 	\hline
			Hyperbolic Umbilic                     &        $(z,w, x^2 + y^2 + zx + wy, xy)$      \\ 	\hline
			Parabolic Umbilic                       &    $(z,w, xy + xz, x^2 + y^3 + yw  )$ \\ \hline         
		\end{tabular}\caption{Generic singularities of exact normal congruences}\label{table1}
	\end{table}
\end{teo}
\begin{proof}
	See theorem 5.2 in \cite{Izumiya} or chapters 4 and 5 in \cite{Livro}.	
	\end{proof}

Now, we define a natural projection
$
P: EN\left(U, \R^4 \times \left(\R^4 \setminus \lbrace \textbf{0} \rbrace \right)\right) \rightarrow Emb(U, \R^4),
$
given by $P(\br,\bx) = \br$. Then, we have the following corollary, which provides a classification of the generic singularities of 3-parameter exact normal congruences.
\begin{cor}\label{corolario5.3}
There is an open dense subset $O \subset  EN\left(U, \R^4 \times \left(\R^4 \setminus \lbrace \textbf{0} \rbrace \right)\right)$, such that the germ of an exact normal congruence $F_{(\br,\bx)}$ at any point $(u_{0}, t_{0}) \in U \times I$ is a Lagrangian stable map germ, for all $(\br,\bx) \in O$.	
\end{cor}

\begin{proof}
	It follows from the fact that $
	P: EN\left(U, \R^4 \times \left(\R^4 \setminus \lbrace \textbf{0} \rbrace \right)\right) \rightarrow Emb(U, \R^4)
$
	 is an open continuous map and from theorem (\ref{teo5.1}).
	\end{proof}

Let us consider some aspects of Lagrangian singularities (see chapter 5 in \cite{Livro}). Take the cotangent bundle $\pi: T^{*}\R^4 \rightarrow \R^4$, whose symplectic structure is given locally by the 2-form $\omega = -d\lambda$, where $\lambda$ is the Liouville 1-form, given locally by $\lambda = \sum_{i=1}^{4}p_{i}dz_{i}$, where $(z_{1}, z_{2}, z_{3}, z_{4}, p_{1}, p_{2}, p_{3}, p_{4})$ are the cotangent coordinates. For a given congruence $F_{(\br,\bx)}$, we define a smooth map $L_{(\br,\bx)}: U \times I \rightarrow T^{*}\R^4 \simeq \R^4 \times (\R^4)^{*}$, given by
\begin{align*}
L_{(\br,\bx)}(u,t) = \left( \br(u) + t \dfrac{\bx}{\lVert \bx \rVert}(u), \dfrac{\bx}{\lVert \bx \rVert}(u)  \right).
\end{align*}

\begin{defi}\normalfont
	We say that $F_{(\br, \bx)}$ is a \textit{Lagrangian Line Congruence} if $L_{(\br,\bx)}$ is a Lagrangian immersion.
\end{defi}

\begin{prop}
		Suppose that $L_{(\br, \bx)}$ is an immersion. Then $F_{(\br, \bx)}$ is a Lagrangian congruence if, and only if, is a normal congruence
\end{prop}

\begin{proof}
Locally, the Liouville $1$-form of $T^{*}\R^4$ is given by  $\lambda = \sum_{i=1}^{4}p_{i}dz_{i}$. So,
\begin{align*}
L^{*}_{(\br,\bx)}(\lambda) &= \sum_{i=1}^{4}\left( \dfrac{\xi_{i}}{\lVert \bx \rVert}(u)dx_{i}(u)  + t\dfrac{\xi_{i}}{\lVert \bx \rVert}(u) d\dfrac{\xi_{i}}{\lVert \bx \rVert}(u) \right) + dt,
\end{align*}
 Therefore, being $\omega = -d\lambda$, we have
\begin{align*}
-L^{*}_{(\br,\bx)}(\omega) = dL^{*}_{(\br,\bx)}(\lambda) &=  \sum_{i=1}^{4}\left( d\dfrac{\xi_{i}}{\lVert \bx \rVert}(u) \wedge dx_{i}(u)  + \dfrac{\xi_{i}}{\lVert \bx \rVert}(u)dt \wedge d\dfrac{\xi_{i}}{\lVert \bx \rVert}(u) \right) \\
&= \left(  \left\langle \left( \dfrac{\bx}{\lVert \bx \rVert} \right)_{u_1}, \br_{u_2}  \right\rangle -  \left\langle \left( \dfrac{\bx}{\lVert \bx \rVert} \right)_{u_2}, \br_{u_1}  \right\rangle \right)du_1 \wedge du_{2} + \\
&+ \left(  \left\langle \left( \dfrac{\bx}{\lVert \bx \rVert} \right)_{u_1}, \br_{u_3}  \right\rangle -  \left\langle \left( \dfrac{\bx}{\lVert \bx \rVert} \right)_{u_3}, \br_{u_1}  \right\rangle \right)du_{1} \wedge du_{3} +\\
&+ \left(  \left\langle \left( \dfrac{\bx}{\lVert \bx \rVert} \right)_{u_2}, \br_{u_3}  \right\rangle -  \left\langle \left( \dfrac{\bx}{\lVert \bx \rVert} \right)_{u_3}, \br_{u_2}  \right\rangle \right)du_{2} \wedge du_{3} + \\
&+ \sum_{i=1}^{3}\left\langle \dfrac{\bx}{\lVert \bx \rVert},  \left( \dfrac{\bx}{\lVert \bx \rVert} \right)_{u_i} \right\rangle dt \wedge du_{i},
\end{align*}
where $\br(u) = (x_{1}(u), x_{2}(u), x_{3}(u), x_{4}(u))$ e $\bx(u) = (\xi_{1}(u), \xi_{2}(u), \xi_{3}(u), \xi_{4}(u))$. Thus
\begin{align*}
-L^{*}_{(\br,\bx)}(\omega) &= \left(  \left\langle \left( \dfrac{\bx}{\lVert \bx \rVert} \right)_{u_1}, \br_{u_2}  \right\rangle -  \left\langle \left( \dfrac{\bx}{\lVert \bx \rVert} \right)_{u_2}, \br_{u_1}  \right\rangle \right)du_1 \wedge du_{2} + \\
&+ \left(  \left\langle \left( \dfrac{\bx}{\lVert \bx \rVert} \right)_{u_1}, \br_{u_3}  \right\rangle -  \left\langle \left( \dfrac{\bx}{\lVert \bx \rVert} \right)_{u_3}, \br_{u_1}  \right\rangle \right)du_1 \wedge du_{3} +\\
&+ \left(  \left\langle \left( \dfrac{\bx}{\lVert \bx \rVert} \right)_{u_2}, \br_{u_3}  \right\rangle -  \left\langle \left( \dfrac{\bx}{\lVert \bx \rVert} \right)_{u_3}, \br_{u_2}  \right\rangle \right)du_2 \wedge du_{3}.
\end{align*}	
	
	Therefore, $L^{*}_{(x,e)}(\omega) = 0$ if, and only if, 
	\begin{align*}
	h_{21} = \left\langle \left( \dfrac{\bx}{\lVert \bx \rVert} \right)_{u_1}, \br_{u_2}  \right\rangle &= \left\langle \left( \dfrac{\bx}{\lVert \bx \rVert} \right)_{u_2}, \br_{u_1}  \right\rangle  = h_{12}\\
	h_{31} = \left\langle \left( \dfrac{\bx}{\lVert \bx \rVert} \right)_{u_1}, \br_{u_3}  \right\rangle &=  \left\langle \left( \dfrac{\bx}{\lVert \bx \rVert} \right)_{u_3}, \br_{u_1}  \right\rangle = h_{13}\\
	h_{32}=\left\langle \left( \dfrac{\bx}{\lVert \bx \rVert} \right)_{u_2},\br_{u_3}  \right\rangle &=  \left\langle \left( \dfrac{\bx}{\lVert \bx \rVert} \right)_{u_3}, \br_{u_2}  \right\rangle = h_{23}. 
	\end{align*}

\end{proof}

By proposition (\ref{propnormal}), we can regard the space of the Lagrangian congruences as follows. A line congruence $F_{(\br, \bx)}$ is a Lagrangian congruence if, and only if, there is a smooth function $t: U \rightarrow \R$, such that $\br(u) + t(u)\bx(u)$ is an immersion and the following conditions hold
\begin{align}\label{eqnormal}
\begin{cases}
t_{u_1}(u) + \left \langle \dfrac{\bx}{\lVert \bx  \rVert}(u), \br_{u_1}(u) \right \rangle = 0\\
t_{u_2}(u) + \left \langle \dfrac{\bx}{\lVert \bx  \rVert}(u), \br_{u_2}(u) \right \rangle = 0\\
t_{u_3}(u) + \left \langle \dfrac{\bx}{\lVert \bx  \rVert}(u), \br_{u_3}(u) \right \rangle = 0.
\end{cases}
\end{align}
So, we can define the space of the Lagrangian congruences
\begin{align*}
L(U, \R^4 \times \left( \R^4 \setminus \lbrace  \textbf{0} \rbrace \right)) = \left\lbrace (\br,t,\bx): \br(u) + t(u)\bx(u)\; \text{is an immersion and}\; (\ref{eqnormal})\; \text{holds} \right\rbrace
\end{align*}
with the Whitney $C^{\infty}$-topology.
 Our idea now is to show that the generic singularities of normal congruences are the same as the generic singularities of exact normal congruences, so, let us define the map
\begin{align*}
T_{rp}: C^{\infty}(U, \R^4 \times \R \times \left( \R^4 \setminus \lbrace  \textbf{0} \rbrace \right)) &\rightarrow C^{\infty}(U, \R^4 \times \left( \R^4 \setminus \lbrace  \textbf{0} \rbrace \right))\\
(\br(u), t(u), \bx(u)) &\mapsto (\br(u) + t(u)\bx(u), \bx(u)).
\end{align*}
\begin{prop}\label{prop5.3}
 $T_{rp}$ is an open continuous map under the Whitney $C^{\infty}$-topology.
\end{prop}
\begin{proof}
	See proposition 5.6 in \cite{Izumiya}.
	\end{proof}

Now, take
\begin{align*}
N(U, \R^4 \times \left( \R^4 \setminus \lbrace \textbf{0} \rbrace \right)) = T_{rp}\left( L\left(U, \R^4 \times \left( \R^4 \setminus \lbrace  \textbf{0} \rbrace \right)\right) \right) \subset  C^{\infty}\left(U, \R^4 \times \left( \R^4 \setminus \lbrace  \textbf{0} \rbrace \right)\right),
\end{align*}
with the Whitney $C^{\infty}$-topology induced from $C^{\infty}\left(U, \R^4 \times \left( \R^4 \setminus \lbrace  \textbf{0} \rbrace \right)\right)$. Note that we can regard $N(U, \R^4 \times \left( \R^4 \setminus \lbrace \textbf{0} \rbrace \right))$ as the space of the normal congruences. Then, we have the following theorem.

\begin{teo}\label{teo5.2}
	There is an open dense set $O' \subset N(U, \R^4 \times \left( \R^4 \setminus \lbrace \textbf{0} \rbrace \right)) $, such that the germ of normal congruence $F_{\left(\br,\bx \right)}$ at any point $(u_0, t_0)$ is a Lagrangian stable germ, for any $(\br,\bx) \in O'$.		
\end{teo}
\begin{proof}
		From  Corollary (\ref{corolario5.3}), there is an open dense subset $O \subset EN\left(U, \R^4 \times \left(\R^4 \setminus \lbrace \textbf{0} \rbrace \right)\right)$, such that the germ of exact normal congruence $F_{(\br,\bx)}$  is a Lagrangian stable germ for all $(\br,\bx) \in O$ at any point $(u_0, t_0) \in U \times I$. As we know, $T_{rp}$ is an open map, so we just need to take $O' = T_{rp}(O)$.
\end{proof}

\section{Blaschke normal congruences}
In this section we deal with one of the most important classes of equiaffine line congruences, which is the class of Blaschke normal congruences. Our goal is to provide a positive answer to the following conjecture, presented in \cite{Izumiya}:
\begin{conjecture}
	Germs of generic Blaschke affine normal congruences at any point are Lagrangian
	stable.
\end{conjecture}
Taking this into account, let us regard $\R^4$ as a four-dimensional affine space with volume element given by $\omega(e_{1}, e_{2}, e_{3}, e_{4}) = \det(e_{1}, e_{2}, e_{3}, e_{4})$, where $\lbrace e_{1}, e_{2}, e_{3}, e_{4}  \rbrace$ is the standard basis of $\R^4$. Let $D$ be the standard flat connection on $\R^4$, thus $\omega$ is a parallel volume element. Let $\br: U \rightarrow \R^4$ be a regular hypersurface with $\br(U) = M$ and $\bx: U \rightarrow \R^4 \setminus \lbrace \bm{0} \rbrace$ a vector field which is transversal to $M$. Thus, decompose the tangent space 
\begin{align*}
T_{p}\R^4 = T_{p}M \oplus \langle \bx(u) \rangle_{\R},
\end{align*}
where $\br(u) = p$. So, it follows that given $X$ and $Y$ vector fields on $M$, we have the decomposition
\begin{align*}
D_{X}Y = \nabla_{X}Y + h(X,Y)\bx,
\end{align*}
where $\nabla$ is the \textit{induced affine connection} and $h$ is the \textit{affine fundamental form} induced by $\bx$, which defines a symmetric bilinear form on each tangent space of $M$. We say that $M$ is \textit{non-degenerate} if $h$ is non-degenerate which is equivalent to say that the Gaussian curvature of $M$ never vanishes (see chapter 3 in \cite{geomAfim}). Using the same idea, we decompose
\begin{align*}
D_{X}\bx = -S(X) + \tau(X)\bx,
\end{align*}
where $S$ is the \textit{shape operator} and $\tau$ is the \textit{transversal connection form}. We say that $\bx$ is an \textit{equiaffine transversal vector field} if $\tau =0$, i.e $D_{X}\bx$ is tangent to $M$.

Using the volume element $\omega$ and the transversal vector field $\bx$, we induce a volume element $\theta $ on $M$ as follows
\begin{align*}
\theta(X, Y,Z) = \omega(X,Y,Z, \bx),
\end{align*}
where $X$, $Y$ and $Z$ are tangent to $M$.

Given a non-degenerate hypersurface $\br: U \rightarrow \R^4$ and a vector field $\bx: U \rightarrow \R^4 \setminus \lbrace \bm{0} \rbrace$ which is transversal to $M = \br(U)$, we take the line congruence generated by $(\br, \bx)$ and the map
\begin{align*}
F_{(\br, \bx)}: U \times I &\rightarrow \R^4\\
(u,t) &\mapsto \br(u) + t\bx(u),
\end{align*}
where $I$ is an open interval.
\begin{defi}\normalfont
	A point $p = F(u,t)$ is called a focal point of multiplicity $m>0$ if the differential $dF$ has nullity $m$ at $(u,t)$.
\end{defi}
The next proposition relates the shape operator $S$ and the above definition.
\begin{prop}{\rm{(\cite{cecil}, Proposition 1)}}
	Let $\br: U \rightarrow \R^4$ be a non-degenerate hypersurface with transversal equiaffine vector field $\bx$. Let $S$ be the shape operator related to $M$ and $\bx$. A point $p = F(u,t)$ is a focal point of $M$ of
multiplicity $m>0$ if and only if $1/t$ is an eigenvalue of $S$ with eigenspace of
	dimension $m$ at $u$.
\end{prop}

For each $u \in U$ and $p \in \R^4$, we decompose $p-\br(u)$ into tangential
and transversal components as follows
\begin{align}\label{funcaosuporte}
p - \br(u) = v(u) + \rho_{p}(u)\bx(u),
\end{align}
where $v(u) \in T_{\br(u)}M$. The real function $\rho_{p}$ is called an \textit{affine support function} associated to $M$ and $\bx$. If we ­fix an Euclidean inner product $\langle .\,, . \rangle$ in $\R^4$, the support function is given by

\begin{align}\label{funcsup}
\rho_{p}(u) = \left \langle p - \br(u), \dfrac{\bx}{\lVert \bx \rVert^2}(u) \right\rangle - \left\langle v(u), \dfrac{\bx}{\lVert \bx \rVert^2}(u) \right\rangle,
\end{align}
thus
\begin{align*}
 \dfrac{\partial \rho}{\partial p_{i}}(u) &= \dfrac{\bx_{i}}{\lVert \bx \rVert^{2}}(u)\label{eq300}.
\end{align*}

\begin{prop}{\rm{(\cite{cecil}, Proposition 2)}}\label{prop2sec2}
	Let $\br: U \rightarrow \R^4$ be a non-degenerate hypersurface and $\bx$ an equiaffine transversal vector field. Then
	\begin{enumerate}[(a)]
		\item \label{itema} The affine support function $\rho_{p}$ has a critical point at $u$ if and only if $p-\br(u)$ is a multiple of $\bx(u)$.
		\item If $u$ is a critical point of $\rho_{p}$, then the Hessian of $\rho_{p}$ at u has the form
			\begin{align*}
		H(X,Y) = h(X, (I - \rho_{p}(u)S)Y), \; X,\, Y \in T_{\br(u)}M.
		\end{align*}
		\item A critical point $u$ of the function $\rho_{p}$ is degenerate if and only if $p$ is a focal
		point of $M$.
	\end{enumerate}
\end{prop}

\begin{obs}\label{obsconjcat}
It follows from item \ref{itema} that the catastrophe set of $\rho$, which is also called the \textit{Criminant set} of $\rho$, is
	\begin{align*}
	C_{\rho} = \lbrace (u,p): p = \br(u) + t\bx(u),\, for\, some\, t \in \R \rbrace.
	\end{align*}

\end{obs}

\begin{prop}
		Let $\br: U \rightarrow \R^4$ be a non-degenerate hypersurface with transversal equiaffine vector field $\bx$. Then the family of germs of  functions $\rho: \left(U \times \R^4, (u_{0}, p_{0})\right) \rightarrow (\R,t_{0})$, where $t_{0} = \rho(u_{0}, p_{0})$ and $u_{0}$ is a critical point of $\rho_{p_{0}}$ is a Morse family of functions.
\end{prop}
\begin{proof} Let us denote $(u,p) = (u_{1}, u_{2}, u_{3}, p_{1}, p_{2}, p_{3}, p_{4})$. In order to prove that $\rho$ is a Morse family we need to prove that the map germ
$\Delta: \left(U \times \R^4, (u_{0}, p_{0})\right) \rightarrow \R^{3}$, given by
\begin{align*}
\Delta \rho(u, p) = \left( \dfrac{\partial \rho}{\partial u_{1}}, \dfrac{\partial \rho}{\partial u_{2}}, \dfrac{\partial \rho}{\partial u_{3}}  \right)(u,p)
\end{align*}
is not singular. Its jacobian matrix is given by

\begin{align}
J(\Delta\rho)(u_{0}, p_{0}) = \left[\begin{array}{cccc}
\dfrac{\partial^{2}\rho_{p_{0}}}{\partial u_{1}\partial u_{1}} & \dfrac{\partial^{2}\rho_{p_{0}}}{\partial u_{1}\partial u_{2}} & \dfrac{\partial^{2}\rho_{p_{0}}}{\partial u_{1}\partial u_{3}} & \frac{1}{\lVert \bx \rVert^{2}} \bx_{u_{1}} - \frac{2 \langle \bx, \bx_{u_{1}}\rangle}{\lVert \bx  \rVert^4}\bx \\ 
\dfrac{\partial^{2}\rho_{p_{0}}}{\partial u_{1}\partial u_{2}} & \dfrac{\partial^{2}\rho_{p_{0}}}{\partial u_{2}\partial u_{2}} & \dfrac{\partial^{2}\rho_{p_{0}}}{\partial u_{2}\partial u_{3}} & \frac{1}{\lVert \bx \rVert^{2}} \bx_{u_{2}} - \frac{2 \langle \bx, \bx_{u_{2}}\rangle}{\lVert \bx  \rVert^4}\bx \\ 
\dfrac{\partial^{2}\rho_{p_{0}}}{\partial u_{1}\partial u_{3}} & \dfrac{\partial^{2}\rho_{p_{0}}}{\partial u_{2}\partial u_{3}} & \dfrac{\partial^{2}\rho_{p_{0}}}{\partial u_{3}\partial u_{3}} & \frac{1}{\lVert \bx \rVert^{2}} \bx_{u_{3}} - \frac{2 \langle \bx, \bx_{u_{3}}\rangle}{\lVert \bx  \rVert^4}\bx
\end{array}\right]_{3 \times 7}. 
\end{align}
If $u_{0}$ is a non-degenerate critical point of $\rho_{p_{0}}$, then $\rank(Hess(\rho_{p_{0}})(u_{0})) = 3$ and the map germ $\Delta\rho$  is not singular. Thus, we just need to check the case in which $u_{0}$ is a degenerate critical point.
\begin{enumerate}
		\item $\bm{\rank Hess(\rho_{p_{0}})(u_{0}) = 0}$ \\In this case, using proposition (\ref{prop2sec2}), we obtain that the eigenspace associated to the eigenvalue $\frac{1}{\rho_{p_{0}}}$ has dimension $3$, hence, the matrix of the shape operator has $\rank$ $3$ and considering that the $\bx$ is equiaffine, $J(\Delta\rho)(u_{0}, p_{0})$ has $\rank$ $3$.
		
\item $\bm{\rank Hess(\rho_{p_{0}})(u_{0}) = 1}$\\In this case, there are two linearly independent vectors $Y,\,Z \in T_{x(u_{0})}M$, such that $H(X,Y) = H(X,Z)=0$, for all $X \in T_{x(u_{0})}M$. Hence, as seen in proposition (\ref{prop2sec2}), the vectors $Y$ and $Z$ are eigenvectors of the shape operator $S$, with eigenvalue $\frac{1}{\rho_{p_{0}}(u_{0})}$. Notice that $\lbrace \br_{u_{1}}(u_{0}), \br_{u_{2}}(u_{0}), \br_{u_{3}}(u_{0}) \rbrace$ is a set of linearly independent vectors and one of these vectors form a basis of $T_{x(u_{0})}M$ together with $Y$ and $Z$. Let us say that $\beta = \lbrace \br_{u_{1}}(u_{0}), Y, Z \rbrace$ is a basis of $T_{\br(u_{0})}M $ (the other cases are analogues). Thus, we can write
\begin{align}\label{vettang}
\br_{u_{2}}(u_{0}) &= a_{1}\br_{u_{1}} + a_{2}Y + a_{3}Z\\
\br_{u_{3}}(u_{0}) &= b_{1}\br_{u_{1}} + b_{2}Y + b_{3}Z \label{vettang1}
\end{align}
which implies that
\begin{align}\label{jacob}
J(\Delta\rho)(u_{0}, p_{0}) = \left[\begin{array}{cccc}
H( \br_{u_{1}}, \br_{u_{1}}) & a_{1}H(\br_{u_{1}}, \br_{u_{1}}) &
b_{1}H(\br_{u_{1}}, \br_{u_{1}}) &
\frac{1}{\lVert \bx \rVert^{2}} \bx_{u_{1}} - \frac{2 \langle \bx, \bx_{u_{1}}\rangle}{\lVert \bx  \rVert^4}\bx \\ 
a_{1}H( \br_{u_{1}}, \br_{u_{1}}) & a_{1}^2H(\br_{u_{1}}, \br_{u_{1}}) &
a_{1}b_{1}H(\br_{u_{1}}, \br_{u_{1}}) & \frac{1}{\lVert \bx \rVert^{2}} \bx_{u_{2}} - \frac{2 \langle \bx, \bx_{u_{2}}\rangle}{\lVert \bx  \rVert^4}\bx\\ 
b_{1}H( \br_{u_{1}}, \br_{u_{1}}) & a_{1}b_{1}H(\br_{u_{1}}, \br_{u_{1}}) &
b_{1}^2H(\br_{u_{1}}, \br_{u_{1}}) & \frac{1}{\lVert \bx \rVert^{2}} \bx_{u_{3}} - \frac{2 \langle \bx, \bx_{u_{3}}\rangle}{\lVert \bx  \rVert^4}\bx
\end{array}\right], 
\end{align} 
where $H(\br_{u_{1}}, \br_{u_{1}}) \neq 0$, since the hessian matrix has $\rank$ $1$.
It follows from the fact that the shape operator $S$ has two linearly independent eigenvectors with nonzero eigenvalue that its $\rank$ is at least $2$, so in the set $\lbrace \bx_{u_1}, \bx_{u_2}, \bx_{u_3} \rbrace$ two of these vectors need to be linearly independent. It is sufficient to  analyze the case when $ \bx_{u_{1}}$ and $\bx_{u_{2} }$ are linearly independent, the other subcases are similar.

\subitem \textbf{Subcase}: $\lbrace \bx_{u_{1}}, \bx_{u_{2}}\rbrace $ linearly independent\\
First of all, if $ \bx_{u_{1}}$ and $\bx_{u_{2}}$ are linearly independent and $\bx$ is equiaffine, then $\bx_{u_{1}} - \frac{2 \langle \bx, \bx_{u_{1}}\rangle}{\lVert \bx  \rVert^4}\bx$ and $\bx_{u_{2}} - \frac{2 \langle \bx, \bx_{u_{2}}\rangle}{\lVert \bx  \rVert^4}\bx $ are linearly independent. Thus, the only case when $J(\Delta\rho)(u_{0}, p_{0})$ has $\rank$ less than $3$ is when its third line is a linear combination of the first and the second lines. Then the same occurs with the Hessian matrix of $\rho_{p_{0}}$ and if we call $L_{1},\, L_{2}$ and $L_{3}$ the lines of this matrix, we have 
\begin{align*}
L_{3} = \lambda L_{1} + \gamma L_{2}, \; where\, \lambda,\, \gamma \in \R.
\end{align*}
  But we know that  $L_{2} = a_{1}L_{1}$ and $L_{3} = b_{1}L_{1}$ and using the above equation
  \begin{align}\label{coef}
  b_{1} = \lambda + \gamma a_{1}.
  \end{align}
  By considering the same combination on the block $3\times 4$ on the right, we have
  \begin{align*}
  \frac{1}{\lVert \bx \rVert^{2}} \bx_{u_{3}} - \frac{2 \langle \bx, \bx_{u_{3}}\rangle}{\lVert \bx  \rVert^4}\bx = \lambda \left( \frac{1}{\lVert \bx \rVert^{2}} \bx_{u_{1}} - \frac{2 \langle \bx, \bx_{u_{1}}\rangle}{\lVert \bx  \rVert^4}\bx  \right) + \gamma \left( \frac{1}{\lVert \bx \rVert^{2}} \bx_{u_{2}} - \frac{2 \langle \bx, \bx_{u_{2}}\rangle}{\lVert \bx  \rVert^4}\bx  \right).
    \end{align*}
Using (\ref{coef}), $\lambda = b_{1} - a_{1} \gamma$ and
\begin{align*}
\frac{1}{\lVert \bx \rVert^{2}} \bx_{u_{3}} - \frac{2 \langle \bx, \bx_{u_{3}}\rangle}{\lVert \bx  \rVert^4}\bx = \left(  b_{1} - a_{1} \gamma \right) \left( \frac{1}{\lVert \bx \rVert^{2}} \bx_{u_{1}} - \frac{2 \langle \bx, \bx_{u_{1}}\rangle}{\lVert \bx  \rVert^4}\bx  \right) + \gamma \left( \frac{1}{\lVert \bx \rVert^{2}} \bx_{u_{2}} - \frac{2 \langle \bx, \bx_{u_{2}}\rangle}{\lVert \bx  \rVert^4}\bx  \right),
\end{align*}
consequently 
\begin{align*}
\frac{1}{\lVert \bx \rVert^{2}} \bx_{u_{3}} -  \left(  b_{1} - a_{1} \gamma \right)\frac{1}{\lVert \bx \rVert^{2}} \bx_{u_{1}} - \gamma \frac{1}{\lVert \bx \rVert^{2}} \bx_{u_{2}} \in TM \cap <\bx> = \lbrace \bm{0} \rbrace,
\end{align*}
thus $\bx_{u_{3}} = (b_{1} - a_{1}\gamma)\bx_{u_{1}} + \gamma \bx_{u_{2}}$. We know that $\bx_{u_{i}} = -S(\br_{u_{i}})$ and from (\ref{vettang1})
\begin{align*}
b_{1}\bx_{u_{1}} - \frac{b_{2}}{\rho_{p_{0}}}Y -  \frac{b_{3}}{\rho_{p_{0}}}Z = (b_{1} - a_{1}\gamma)\bx_{u_{1}} + \gamma \left(a_{1}\bx_{u_1} - \frac{a_{2}}{\rho_{p_{0}}}Y -  \frac{a_{3}}{\rho_{p_{0}}}Z\right),
\end{align*}
therefore,
\begin{align}\label{eqetapa}
-\frac{b_{2}}{\rho_{p_{0}}}Y -  \frac{b_{3}}{\rho_{p_{0}}}Z =  -\gamma \frac{a_{2}}{\rho_{p_{0}}}Y - \gamma \frac{a_{3}}{\rho_{p_{0}}}Z.
\end{align}
Then, 
\begin{align*}
a_{2}\gamma &= b_{2}\\
a_{3}\gamma &= b_{3}.
\end{align*}
Finally
\begin{align*}
\gamma \bm{x}_{u_{2}} &= a_{1}\gamma \bm{x}_{u_{1}} + a_{2}\gamma Y + a_{3} \gamma Z\\
&= \left(-\lambda + b_{1}\right)\bm{x}_{u_{1}} + b_{2}Y + b_{3}Z \\
&= -\lambda \bm{x}_{u_{1}} + \bm{x}_{u_{3}}.
\end{align*}
But this contradicts the fact that $\lbrace \bm{x}_{u_{1}}, \bm{x}_{u_{2}}, \bm{x}_{u_{3}} \rbrace$ are linearly independent.

\item $\bm{\rank Hess(\rho_{p_{0}})(u_{0}) = 2}$\\
In this case, there is $Y \in T_{\bm{x}(u_{0})}M$ eigenvector of the shape operator $S$ with eigenvalue $\frac{1}{\rho_{p_{0}}(u_{0})}$, by proposition (\ref{prop2sec2}). $\rank Hess(\rho_{p_{0}})(u_{0}) = 2$, then it follows that at least two of the vectors  $\bm{x}_{u_{i}}$, $i=1,2,3$ do not belong to the eigenspace of $\frac{1}{\rho_{p_{0}}(u_{0})}$, otherwise $\rank Hess(\rho_{p_{0}})(u_{0}) < 2$, by proposition (\ref{prop2sec2}). If we look at $\lbrace \bm{x}_{u_{1}}(u_{0}), \bm{x}_{u_{2}}(u_{0}), Y \rbrace$ as a basis of 
$T_{\bm{x}(u_{0})}M$ (the other cases are analogous) and write (in $u_{0}$)
\begin{align*}
\bm{x}_{u_{3}} = a_{1}\bm{x}_{u_{1}} + a_{2}\bm{x}_{u_{2}} + a_{3}Y,
\end{align*}
this case follows in a similar way to the last one.
\end{enumerate}

	\end{proof}
\begin{obs}\normalfont\label{obsmapalag4}
	It follows from the above proposition that the 4-parameter family of germs of functions  $\rho: \left(U \times \R^4, (u_{0}, p_{0})\right) \rightarrow (\R,t_0)$, where $u_{0}$ is a critical point of $\rho_{p_{0}}$, is a Morse family. Furthermore, if $p_{0} = \bm{x}(u_{0}) + t_{0}\bx(u_{0})$ (where $t_{0} = \rho_{p_{0}}(u_{0})$), the Lagrangian immersion associated to this Morse family is $L: \left( U \times \R, (u_{0}, t_{0}) \right) \rightarrow T^{*}\R^4$, given by
	\begin{align*}
	L(u,t) = \left( \bm{x}(u) + t\bx(u), \dfrac{\bx}{\lVert \bx \rVert^{2}}(u)  \right),
	\end{align*}
	whose Lagrangian map associated is $F_{(\bm{x}, \bx)} = \pi \circ L(u,t) = \bm{x}(u) + t\bx(u)$, where $\pi: T^{*}\R^4 \rightarrow \R^4$.
\end{obs}

\begin{defi}\normalfont
		Let $\bm{x}: U \rightarrow \R^4$, with $\bm{x}(U)=M$, be a non-degenerate hypersurface and take $\bx: U \rightarrow (\R^4 \setminus \lbrace \bm{0} \rbrace)$ an equiaffine transversal vector field. Define $\bnu: U \rightarrow (\R^4 \setminus \lbrace \bm{0} \rbrace)$, such that for each $\br(u) = p \in M$ and $v \in T_{p}(M)$
	\begin{align}\label{mapaconormal}
	\langle \bnu(u), \bx(u) \rangle =1\; and \; \langle \bnu(u), v \rangle =0.
	\end{align}
	Each $\bnu(u)$ is called the \textit{conormal vector} of $\br$ relative to $\bx$ at $p$. The map $\bnu$ is called the \textit{conormal map}.
\end{defi}
\begin{obs}\normalfont
	Using (\ref{funcaosuporte}) and (\ref{mapaconormal}), we obtain
	\begin{align*}
	\rho_{p}(u) = \langle p-\br(u), \bnu(u) \rangle,
	\end{align*}
	where $\rho_{p}$ is the affine support function.
\end{obs}

\subsection{Blaschke Exact Normal Congruences}\ \  \\
	Given a non-degenerate hypersurface $\br(U) = M$, we know that the affine fundamental form $h$ is non-degenerate, then it can be treated as a non-degenerate metric (not necessarily positive-definite) on $M$.
\begin{defi}\normalfont
	Let $\br: U \rightarrow \R^4$ be a non-degenerate hypersurface. A transversal vector field $\bx: U \rightarrow \R^4 \setminus \lbrace \bm{0} \rbrace$ satisfying
	\begin{enumerate}
		\item $\bx$ is equiaffine.
		\item The induced volume element $\theta$ coincides with the volume element $\omega_{h}$ of the non-degenerate metric $h$.
	\end{enumerate}
is called \textit{the Blaschke normal vector field} of $M$.

\end{defi}

Let $Emb_{ng}(U, \R^4) = \lbrace \br: U \rightarrow \R^4: \br\; \text{is a non-degenerate embedding} \rbrace$ be the space of non-degenerate regular hypersurfaces with the Whitney $C^{\infty}$- topology. Define the space of
the Blaschke exact normal congruences as
\begin{align*}
BEN(U, \R^4 \times (\R^4 \setminus \lbrace \bm{0} \rbrace)) = \left\lbrace (\bm{x}, \bx): \bm{x} \in Emb_{ng}(U, \R^4),\; \bx\; \text{is the} \right.\\
\left. \text{Blaschke normal vector field of}\; \bm{x} \right\rbrace.
\end{align*}

\begin{obs}\normalfont\label{observ5.3}
	Given a non-degenerate hypersurface $\br(U) = M$, its Blaschke vector field is unique up to sign and is given by
\begin{align}\label{vetorBlaschke}
\bx(u) = \left|K(u)\right|^{1/4}N(u) + Z(u), 
\end{align}
where $K$ is the Gaussian curvature of $M$, $N$  its unit normal and $Z$ is a vector field on $M$, such that
\begin{align*}
II(Z, X) = -X(\left| K \right|^{1/4}), \forall\, X \in TM
\end{align*}
where $II$ denotes the second fundamental form of $M$ (for details, see example 3.4 in \cite{geomAfim}). We can write the vector field $Z$ in terms of the coefficients of the second fundamental form and the partial derivatives of $\left| K \right|^{1/4}$ . From (\ref{vetorBlaschke}) it follows that the conormal vector relative to the Blaschke vector field of a non-degenerate hypersurface in $\R^4$ is given by
\begin{align}\label{vetorconormal}
\bnu(u) =  \left|K(u)\right|^{-1/4}N(u)
\end{align}
Then, we identify (with the Whitney $C^{\infty}$-topology) the spaces $Emb_{ng}(U, \R^4)$ and 
\begin{multline*}
S_{con}(U, \R^4 \times \R^4 \setminus \lbrace \bm{0} \rbrace) = \left\lbrace (\bm{x}, \nu) \in C^{\infty}(U, \R^4 \times \R^4 \setminus \lbrace \bm{0} \rbrace): \bm{x} \in Emb_{ng}(U, \R^4)\;\text{and}\; \bnu\; \text{is} \right.\\
\left. \text{the conormal of $\bm{x}$ relative to the Blaschke vector field} \right\rbrace 
\end{multline*}
\end{obs}

\begin{defi}\normalfont
	Let $\bm{x}: U \rightarrow \R^4$, with $\bm{x}(U)=M$, be a non-degenerate hypersurface. We define the \textit{conormal bundle} of $M$ by 
	\begin{align*}
	N^{*}_{\bm{x}} = \lbrace (p,v): p \in M,\; \langle v, w \rangle = 0,\; \forall\; w \in T_{p}M \rbrace \subset T^{*}\R^4.
	\end{align*}
\end{defi}

\begin{obs}\normalfont\label{observ5.4}
	Note that we can look at $S_{con}(U, \R^4 \times \R^4 \setminus \lbrace \bm{0} \rbrace)$ as a section of the conormal bundle of $M$.	
\end{obs}

Let us define the following maps
\begin{align}
H: \left( \R^{4}  \times \R^4 \setminus \lbrace \bm{0} \rbrace \right) \times \R^4 &\rightarrow \R\\
(A, B, C) &\mapsto \langle B, C-A \rangle \nonumber
\end{align}
\begin{align}
g: U &\rightarrow  \R^{4}  \times \R^4 \setminus \lbrace \bm{0} \rbrace\\
u &\mapsto (\bm{x}(u), \bnu(u)),\nonumber
\end{align}
where $g \in S_{con}(U, \R^{4}  \times \R^4 \setminus \lbrace \bm{0} \rbrace)$.
If we fix a parameter $C$,  $H_{C}: \R^{4}  \times \R^4 \setminus \lbrace \bm{0} \rbrace \rightarrow \R$ is a submersion, therefore, $H_{C} \circ g$ is a contact map. 
Finally, note that \begin{align*}\rho(u,p) = H \circ \left(g, I_{d}\big|_{\R^4} \right)(u,p).\end{align*}

\begin{prop}\label{Prop5.4}
For a residual subset of $Emb_{ng}(U, \R^4\times \R^4 \setminus \lbrace \bm{0} \rbrace)$ the family $\rho$ is locally 
$\cP$-$\mathcal{R}^{+}$-versal.
\end{prop}
\begin{proof}
	Following the identification in remark ({\ref{observ5.3}}) and the notation in remark ({\ref{observ5.4}}) we can apply theorem (\ref{teoMontaldiadap}) in order to show that there is a residual subset of $Emb_{ng}(U, \mathbb{R}^4 \times \mathbb{R}^4 \setminus \lbrace 0 \rbrace)$ for which $\rho$ is locally $\mathcal{P}$- ${\cal{R}}^{+}$-versal.
\end{proof}

\begin{teo}
There is a residual subset $O \subset Emb_{ng}(U, \R^4)$ such that the germ of the Blaschke exact normal congruence $F_{(\br, \bx)}$ at any point $(u_{0}, t_{0}) \in U \times I$ is a
Lagrangian stable map germ for any $\br \in O$, i.e., $ \forall\, \br \in O$, $F_{(\br,\bx)}$ is an immersive germ, or $\cA$-equivalent to one of the normal forms in table (\ref{table1}).
\end{teo}
\begin{proof}
	Let us take the map germ $F_{(\bm{x}, \bx)}: (U \times \R, (u_{0}, t_{0})) \rightarrow (\R^4, p_{0})$. Thus $u_{0}$ is a critical point of $\rho_{p_{0}}$, by proposition (\ref{prop2sec2}). Then,  $\rho: \left(U \times \R^3, (u_{0}, p_{0})\right) \rightarrow (\R,t_{0})$ is a Morse family of functions. Furthermore, by Remark (\ref{obsconjcat}), the Lagrangian map related to this family is $F_{(\bm{x}, \bx)}$. It is known that if $\rho$ is $\cP$-${\cal{R}}^{+}$-versal, then $F_{(\bm{x}, \bx)}$ is Lagrangian stable (see Theorem 5.4 in \cite{Livro}), so the result follows from proposition (\ref{Prop5.4}).
	\end{proof}

The map
\begin{align}
\Pi: BEN\left(U, \R^4 \times \left(\R^4 \setminus \lbrace \textbf{0} \rbrace \right)\right) \rightarrow Emb_{ng}(U, \R^4),
\end{align}
given by $\Pi(\br, \bx) = \br$, is open and continuous. Using this, we obtain the following corollary.

\begin{cor}\label{cor5.1}
	There is a residual subset $\cO \subset  BEN\left(U, \R^4 \times \left(\R^4 \setminus \lbrace \bm{0} \rbrace \right)\right)$, such that the germ of the Blaschke exact normal congruence $F_{(\br, \bx)}$ at any point $(u_{0}, t_{0}) \in U \times I$ is a
Lagrangian stable map germ for any $(\br, \bx) \in \cO$, i.e., $ \forall\, (\br, \bx) \in \cO$, $F_{(\br,\bx)}$ is an immersive germ, or $\cA$-equivalent to one of the normal forms in  table (\ref{table1}).
\end{cor}

\subsection{Blaschke Normal Congruences}\ \  \\
Let
\begin{align*}
BN(U, \R^4 \times (\R^4 \setminus \lbrace \bm{0} \rbrace))=& \left\{  (\bm{x}, \bm{\xi}): \exists\, t \in C^{\infty}(U, \R),\, s.t.\, \bm{y}(u) = \bm{x}(u) + t(u)\bm{\xi}(u)\in Emb_{ng}(U, \R^4)\right. \\       &  \left. \text{and}\, \bm{\xi}\, \text{is the Blaschke normal vector field}\, \text{of}\, \bm{y} \right\}
\end{align*}
be the space of the Blaschke normal congruences. Alternatively we can look at this space as a subspace of $C^{\infty}(U, \R^4 \times \R \times (\R^4 \setminus \lbrace \bm{0} \rbrace))$
	\begin{align*}
BN(U, \R^4 \times \R \times (\R^4 \setminus \lbrace \bm{0} \rbrace)) =& \left\{ (\bm{x}(u), t(u), \bm{\xi}(u)):\bm{y}(u)=\bm{x}(u) + t(u)\bm{\xi}(u)\in Emb_{ng}(U, \R^4)\, \text{and} \right. \\     & \left. \bm{\xi}\, \text{is the Blaschke normal vector field }\, \text{of}\, \bm{y} \right\} 
\end{align*}
In both cases, with the Whitney $C^{\infty}$-topology.

The map
\begin{align*}
T_{rp}: C^{\infty}(U, \R^4 \times \R \times \left( \R^4 \setminus \lbrace  \textbf{0} \rbrace \right)) &\rightarrow C^{\infty}(U, \R^4 \times \left( \R^4 \setminus \lbrace  \textbf{0} \rbrace \right))\\
(\bm{x}(u), t(u), \bx(u)) &\mapsto (\bm{x}(u) + t(u)\bx(u), \bx(u)),
\end{align*}
is open and continuous (see proposition \ref{prop5.3}) in the Whitney $C^{\infty}$-topology. Notice that
\begin{align*}
BEN\left(U, \R^4 \times \left(\R^4 \setminus \lbrace \bm{0} \rbrace \right)\right) \subset C^{\infty}(U, \R^4 \times \R \times \left( \R^4 \setminus \lbrace  \textbf{0} \rbrace \right))
\end{align*}
 with the following identification 
\begin{align*}
BEN\left(U, \R^4 \times \left(\R^4 \setminus \lbrace \bm{0} \rbrace \right)\right) \ni (\bm{x}, \bx) \sim (\bm{x}, 0, \bx),\end{align*} where $\bm{x} \in Emb_{ng}(U, \R^4)$ and $\bx$ is its Blaschke normal vector field. Furthermore, we can look at the space of the Blaschke normal congruences as the space
\begin{align}\label{eqident}
\widetilde{BN}(U, \R^4 \times (\R^4 \setminus \lbrace \bm{0} \rbrace)) = T_{rp}\left( 	BN(U, \R^4 \times \R \times (\R^4 \setminus \lbrace \bm{0} \rbrace)) \right).
\end{align}
Thus, $T_{rp}(BEN\left(U, \R^4 \times \left(\R^4 \setminus \lbrace \bm{0} \rbrace \right)\right)) = \widetilde{BN}\left(U, \R^4 \times \left(\R^4 \setminus \lbrace \bm{0} \rbrace \right)\right)$. Hence, we obtain the following theorem.

\begin{teo}
	There is a residual subset $\cO' \subset \widetilde{BN}(U, \R^4 \times \left( \R^4 \setminus \lbrace \bm{0} \rbrace \right)) $, such that the germ of Blaschke normal congruence $F_{\left(\bm{x},\bm{e} \right)}$ at any point $(u_{0}, t_{0}) \in U \times I$ is a
	Lagrangian stable map germ for any $(\br, \bx) \in \cO'$, i.e., $ \forall\, (\br, \bx) \in \cO'$, $F_{(\br,\bx)}$ is an immersive germ, or $\cA$-equivalent to one of the normal forms in  table (\ref{table1}).
\end{teo}
\begin{proof}
It is known that map $T_{rp}$ is open and continuous and $T_{rp}(BEN\left(U, \R^4 \times \left(\R^4 \setminus \lbrace \bm{0} \rbrace \right)\right)) = \widetilde{BN}\left(U, \R^4 \times \left(\R^4 \setminus \lbrace \bm{0} \rbrace \right)\right)$. If $\cU \subset BEN\left(U, \R^4 \times \left(\R^4 \setminus \lbrace \bm{0} \rbrace \right)\right)$ is open and dense, then its image by $T_{rp}$ is an open dense subset of $\widetilde{BN}\left(U, \R^4 \times \left(\R^4 \setminus \lbrace \bm{0} \rbrace \right)\right)$. Take $\cO = \bigcap\limits_{i \in \NN}\cO_{i}$ the residual subset of $BEN\left(U, \R^4 \times \left(\R^4 \setminus \lbrace \bm{0} \rbrace \right)\right)$ given in Corollary (\ref{cor5.1}). We can show that $T_{rp}(\cO) = \cO' = \bigcap\limits_{i \in \NN}\cO'_{i}$, where $T_{rp}(\cO_{i}) = \cO'_{i}$, therefore $\cO'$ is residual.
\end{proof}

\begin{ex}\normalfont
Taking into account \cite{german}(section 2) and \cite{global}(section 2.2.4) it is possible to parametrize a non-degenerate hypersurface $M$ around an elliptic point, by considering not only $\cal{R}$-equivalence but also affine transformations of $\R^4$, as a graph of a function $h: U \rightarrow \R$, such that
\begin{align}
h(u_{1},u_{2},u_{3}) &= 1/2\,({u_{1}}^{2}+{u_{2}}^{2}+{u_{3}}^{2})+{\it a_{111}}\,u_{1}u_{2}u_{3}+1/6\, \left( -{
	\it a_{120}}-{\it a_{102}} \right) {u_{1}}^{3}+1/2\,{\it a_{210}}\,{u_{1}}^{2}u_{2} \nonumber \\
&+1/2\,{
	\it a_{201}}\,{u_{1}}^{2}u_{3}+1/6\, \left( -{\it a_{210}}-{\it a_{012}} \right) {u_{2}}^{3
}+1/2\,{\it a_{120}}\,u_{1}{u_{2}}^{2}+1/2\,{\it a_{021}}\,{u_{2}}^{2}u_{3} \nonumber \\
&+1/6\, \left( -{
	\it a_{201}}-{\it a_{021}} \right) {u_{3}}^{3}+1/2\,{\it a_{102}}\,u_{1}{u_{3}}^{2}+1/2\,{
	\it a_{012}}\,u_{2}{u_{3}}^{2} + O(3).
 \end{align}

Here $O(3)$ means functions of order higher than 3. Since the group of affine transformations is different from the group of \textit{Euclidean motions} (translations and rotations) it follows that this is not necessarily a local parametrization of $M$ around an Euclidean umbilic point.
Using this parametrization, the Blaschke normal vector of $M$ at the origin is given by $(0,0,0,1)$. If we choose $a_{111}= a_{210} = a_{012} = a_{201} = 0$, $a_{120} = a_{102} = 1$ and $a_{021} = 2$, it follows that
\begin{align*}
h(u_{1},u_{2},u_{3}) = 1/2({u_{1}}^{2}+{u_{2}}^{2}+{u_{3}}^{2})-1/3{u_{1}}^{3}+1/2u_{1}{u_{2}}^{2}+1/2
u_{1}{u_{3}}^{2}+{u_{2}}^{2}u_{3}-1/3{u_{3}}^{3}.
\end{align*}
Using (\ref{vetorBlaschke}) we can compute the Blaschke normal vector field of $M$
\begin{align*}
\bx(u_{1},u_{2},u_{3}) = &(6/5u_{1}+ 18/5u_{1}^2 - 17/5(u_{2}^2 + u_{3})^2 + O(3), 2u_{2} - 6u_{1}u_{2}- 52/5u_{2}u_{3} + O(3), \\ &2u_{3}-6u_{1}u_{3}-26/5(u_{2}^2 - u_{3}^2) + O(3),
1+ 3/5u_{1}^2 + u_{2}^2 + u_{3}^2 + O(3)).
\end{align*}
Furthermore, the congruence map $F_{(\br, \bx)}(u_{1},u_{2}, u_{3},t) = \br(u_{1}, u_{2}, u_{3}) + t \bx(u_{1}, u_{2}, u_{3}) $ has a singular point at  $(0,0,0,-1/2)$ and its 2-jet at this point is given by
\begin{align*}
F_{(\br, \bx)}(u_{1},u_{2},u_{3},t) & = (2/5u_{1}-9/5{u_{1}}^{2}+17/10u_{2}^2+17/10u_{3}^2+6/5 \left( t+1/2 \right) u_{1}, 3u_{1}u_{2}+26/5u_{2}u_{3} \\
&+2 \left( t+1/2
\right) u_{2}, 3u_{1}u_{3}+ 13/5u_{2}^2-13/5u_{3}^2+2 \left( t+1/2 \right) u_{3}, t+1/5{u_{1}}^{2}).
\end{align*}
If we take $\lambda = s+\frac{1}{2} = t + \frac{1}{5}u_{1}^2$, then it is possible to verify that $F_{(\br, \bx)}(u, \lambda)$ is a versal deformation of $f_{0}(u) = (2/5u-9/5{u}^{2}+17/10u_{2}^2+17/10u_{3}^2, 3u_{1}u_{2} + 26/5 u_{2}u_{3}, 3u_{1}\,u_{3}+13/5u_{2}^2 - 13/5u_{3}^2)$, which is an elliptic umbilic  singularity.
\end{ex}

\begin{ex}\normalfont
Let us take a non-degenerate hypersurface given by the graph of
	\begin{align}
h(u_{1}, u_{2}, u_{3}) &= -1/2{u_{1}}^{2}-1/2{u_{2}}^{2}+1/2{u_{3}}^{2}+1/6{u_{1}}^{3}-1
/2{u_{1}}^{2}u_{2} \nonumber\\
&+1/2u_{1}\,{u_{3}}^{2}+1/3{u_{2}}^{3}+1/2u_{
	2}{u_{3}}^{2}.
	\end{align}
	Then, in a similar way to the last example, it is possible to verify that the map $F_{(\br, \bx)}$, where $\br(u_{1}, u_{2}, u_{3}) = (u_{1}, u_{2}, u_{3}, h(u_{1}, u_{2}, u_{3}))$ and $\bx$ is the Blaschke normal vector field of $\br$, has a hyperbolic umbilic singularity at $(0,0,0,5/4)$.
\end{ex}

\begin{ex}\normalfont
By taking a non-degenerate hypersurface given by the graph of
\begin{align}
h(u_{1}, u_{2}, u_{3}) = 1/2(-{u_{1}}^{2}-{u_{2}}^{2}+{u_{3}}^{2})+2u_{1}u_{2}u
_{3}+1/2u_{1}{u_{2}}^{2}+1/2u_{1}{u_{3}}^{2}+1/4{u_{2}}^{4}
\end{align}
it follows, in a similar way to the first example, that the map $F_{(\br, \bx)}$, associated to the Blaschke exact normal congruence, has a parabolic umbilic singularity at $(0,0,0,-5/6)$.
\end{ex}

\vspace{0.5cm}
	Igor Chagas Santos, Instituto de Ciências Matemáticas e de Computacão - Universidade de São Paulo, Av. Trabalhador sao-carlense, 400 - Centro, CEP: 13566-590 - São Carlos - SP, Brazil.\\
	e-mail: igor.chs34@usp.br\\ \\
	Maria Aparecida Soares Ruas, Departamento de Matemática, ICMC Universidade de São Paulo, Campus de São Carlos, Caixa Postal 668, CEP 13560-970, São Carlos-SP, Brazil.\\
	e-mail: maasruas@icmc.usp.br \\ \\
Débora Lopes da Silva,	Departamento de Matemática Universidade Federal do Sergipe Av. Marechal Rondon, s/n Jardim Rosa Elze - CEP 49100-000 São Cristóvao, SE, Brazil. \\
e-mail: deb@deboralopes.mat.br
\end{document}